\documentclass[a4paper]{amsart}
\usepackage{amssymb}
\usepackage{tikz}
\usepackage{epic,curves,mfpic}
\usepackage{hyperref}

\newtheorem{theo}{Theorem}[section]

\theoremstyle{definition}
\newtheorem{defi}[theo]{Definition}

\newcommand{\Z}{\mathbb{Z}}

\theoremstyle{remark}

\begin{document}

\title{Patterns and Tracks}
\author{M.J. Dunwoody }


\begin {abstract}
Patterns in triangulated $2$-spheres and $3$-spheres are investigated.  
 A reworking of
 the proof of Abigail Thompson of the Recognition Algorithm for $S^3$ is presented   together with an explanation of how it can be adapted to give a short, self-contained proof of the Poincar\'e Conjecture.

\end {abstract}

\maketitle
\begin {section} {Introduction}


 The Recognition Algorithm  of Hyam Rubinstein \cite {[R97]} 
 gives a way of deciding if  a compact triangulated  $3$-manifold $M$ is a $3$-sphere. Given such a $3$-manifold one decomposes it by cutting it up along a maximal set of non-parallel, normal, embedded, separating $2$-spheres.  This  will give a finite set of compact $3$-manifolds $M_0$ with at least one boundary component, where each such component  is a $2$-sphere.   Under the condition that $2$-spheres separate, it was  shown by Abigail Thompson in \cite {[T94]} that  each $M_0$ with more than one boundary component is a punctured $3$-ball,  and $M$ is a $3$-sphere if and only if every $M_0$ with one boundary component either contains a vertex of the triangulation or it contains an embedded  almost normal $2$-sphere.    There is a very nice account of this theory by Joel Hass in \cite {JH13}.

A proof of part of this result was given in the first version of this paper, using the language of patterns and tracks developed in \cite {[DD89]}.   This is now extended to a proof of the whole result.
It is then shown, as in \cite {[D25]}, how to adapt this proof to give a proof of the Poincar\'e 
Conjecture.  The Poincar\'e Conjecture was proved by Perelman in 2002.

In 2002 I attempted a proof of the Poincar\' e Conjecture along these lines and a number of
errors were pointed out.   At the time I was unable to resolve all of them and came to the conclusion that the approach
could not work.  Recently I wrote an account of my research \cite {[D24]}, particularly relating to Stallings' Theorem and the accessibility of finitely generated groups and 
that made me think again about my aborted proof.   It now seems to me that the approach was a good one and I 
have come up with this proof that I think resolves the earlier problems.

Slides for a short lecture course on this material are on my home page.

I am very grateful to Peter Kropholler for his interest in my research.

\end {section}

\begin{section} {Patterns and Tracks}

  \begin{figure}[htbp]
\centering
\begin{tikzpicture}[scale=1]

  \draw  (0,.0) -- (4, 0) -- (2, 2) -- (0,0)  ;

\draw [blue](1.5,1.5) -- (2.75,1.25) ;
\draw  [blue](1.25, 1.25) -- (3,1) ;
\draw [blue](.5, 0) -- (.4, .4) ;
\draw [blue](.7, 0) --(.8, .8);
\draw [blue](1.1, 0) -- (1, 1) ;
\draw [blue](3, 0) -- (3.5, .5) ;
\draw  [above,left,,red ] (1.2,1.2) node {$5$} ;
\draw  [above,right,red] (3,1) node {$3$} ;
\draw  [below,red ] (2,0) node {$4$} ;
\end{tikzpicture}

\caption {\label 2}
\end{figure}

   Let $K$ be a finite $2$-complex with polyhedron $|K|$. A {\it pattern}
is a subset $P$ of $|K|$ satisfying the following conditions:-

\begin{itemize}

\item  [(i)] For each $2$-simplex $\sigma $ of $K$,  $P\cap |\sigma|  $ is a union of finitely many disjoint straight lines joining distinct faces of $\sigma$.
  \item [(ii)] For each $1$-simplex $\gamma$ of $K$, $P\cap |\gamma | $ consists of finitely many points in the interior of $|\gamma |$.
\end {itemize}

A {\it track} is a connected pattern.   A pattern is uniquely determined by its intersection with the $1$-skeleton. See Figure 1.

   If two patterns $P$ and $Q$ intersect each $1$-simplex in the same number of points then the patterns are said to be {\it equivalen}t.  Two equivalent disjoint tracks in the same $2$-complex are said to be {\it parallel}.

 \end {section}
\begin {section} {Patterns in $2$-spheres}

If a pattern  in a tetrahedron $T$ is as in Figure 2 then  the tracks are all $3$-tracks or $4$-tracks. A  pattern in a $3$-manifold is called a normal pattern if the intersection with the boundary of every $3$-simplex $\rho $  is like this.  
  
  \begin{figure}[htbp]
\centering
\begin{tikzpicture}[scale=.8]

  \draw  (0,0) -- (0,4) --(4,4)-- (0,0)  ;
  \draw  (0,0) --(4,0) -- (4,4) ;

\draw [red ] (4, 1.5) --(0,1.5) ;
\draw [red ] (4, 1.7) --(0,1.7) ;
\draw [red ] (4, 1.9) --(0,1.9) ;
\draw [red ] (4, 2.1) --(0,2.1) ;
\draw [red ] (4, 2.3) --(0,2.3) ;
\draw [red ] (4, 2.5) --(0,2.5) ;

\draw   (0,0) node {$\bullet $} ;
\draw   (4,0) node {$\bullet $} ;
\draw  (0,4) node {$\bullet $} ;
\draw  (4,4) node {$\bullet $} ;
\
 \draw  (6,0)--(6,4) --(10,0)--(10,4) ; 

\draw   (6,0) node {$\bullet $} ;
\draw   (10,0) node {$\bullet $} ;
\draw  (6,4) node {$\bullet $} ;
\draw  (10,4) node {$\bullet $} ;

\draw [red ] (6, 1.5) --(10,1.5) ;
\draw [red ] (6, 1.7) --(10,1.7) ;
\draw [red ] (6, 1.9) --(10,1.9) ;
\draw [red ] (6, 2.1) --(10,2.1) ;
\draw [red ] (6, 2.3) --(10,2.3) ;
\draw [red ] (6, 2.5) --(10,2.5) ;

\draw [red] (6,2.7)--(7.3,4) ;
\draw [red] (6,3.1)--(6.9,4) ;
\draw [red] (6,3.5)--(6.5,4) ;

\draw [red] (0,2.7)--(1.3,4) ;
\draw [red] (0,3.1)--(0.9,4) ;
\draw [red] (0,3.5)--(0.5,4) ;

\draw [red] (4, 3.8) --(3.8, 4) ;
\draw [red] (4, 3.4) --(3.4, 4) ;
\draw [red] (4, 3) --(3, 4) ;
\draw [red] (4, 2.6) --(2.6,4) ;

\draw [red ] (10, 3.8) --(9.8,4) ;
\draw [red ] (10, 3.4) --(9.4,4) ;
\draw [red ] (10, 3) --(9,4) ;
\draw [red ] (10, 2.6) --(8.6,4) ;

\draw [red] (10,1.3)--(8.7,0) ;
\draw [red] (10,0.9)--(9.1,0) ;
\draw [red] (10,0.5)--(9.5,0) ;

\draw [red] (6,.2)--(6.2,0) ;
\draw [red] (6,1)--(7,0) ;

\draw [red] (6,.6)--(6.6,0) ;

\draw [red] (6,1.4)--(7.4,0) ;
\draw [red] (0,1.4)--(1.4,0) ;
\draw [red] (0,.2)--(0.2,0) ;
\draw [red] (0,1)--(1,0) ;

\draw [red] (0,.6)--(.6,0) ;

\draw [red] (4, 1.3) --(2.7,0) ;
\draw [red] (4, .9) --(3.1,0) ;
\draw [red] (4, .5) --(3.5,0) ;

\

\draw [red] (6,.2)--(6.2,0) ;
\draw [red] (6,1)--(7,0) ;

\draw [red] (6,.6)--(6.6,0) ;

\draw [red] (6,1.4)--(7.4,0) ;
\draw [red] (6,.2)--(6.2,0) ;
\draw [red] (6,1)--(7,0) ;

\draw [red] (6,.6)--(6.6,0) ;

\draw [red] (6,1.4)--(7.4,0) ;

\draw (6,0) --(10,0) ;

\draw (6,4) --(10,4) ;

\draw  [left] (0,0) node {$p$} ;
\draw  [left] (6,0) node {$p$} ;
\draw  [left] (0,4) node {$q$} ;
\draw  [left] (6,4) node {$q$} ;

\draw  [right ] (4,0) node {$s$} ;
\draw  [right ] (10,0) node {$s$} ;
\draw  [right ] (4,4) node {$r$} ;

\draw  [right ] (10,4) node {$r$} ;

\end{tikzpicture}
\caption {Normal Pattern}
\end{figure}
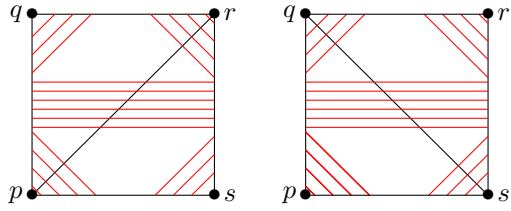
  
  
There are other patterns in a tetrahedron, as we will see later.
   
Two separating tracks are parallel unless the space between them contains a vertex or the centre point of a face.      This is used in the accessibility of finitely presented groups,

A tetrahedron is a triangulation of a $2$-sphere.   There are many triangulations of a $2$-sphere.  Another one is the icosahedron.  Let $K$ be such a $2$-complex.   Every edge ($1$-simplex) belongs to two faces
and the neighbourhood of each vertex is a disc.

Suppose $T$ is a simple closed curve in $K$ that is in general position with respect to the triangulation.   Then $T$ is a track (or rather it can be straightened to be a track) if and only if in the intersection with each $2$-simplex there are no returning arcs, i.e.
 one as in Figures 3 and 4 for which the end points are in the same edge of the triangulation. 
\

 \begin{figure}[htbp]
\centering
\begin{tikzpicture}[scale =1.5]

 \draw  (0,2)--(2,4)--(4,2)--cycle ;

\draw [thick, blue] (1,3)--(1.1,2.7)--(2.9,2.7)--(3,3) ;
\draw [thick, blue] (.8,2.8)--(1,2) ;
\draw [thick, blue] (.5,2.5)--(.6,2) ;
\draw  [thick, blue]  (2.5,2) arc (0:180: .5) ;
\draw [thick, blue] (3.2,2.8)--(3,2) ;
\draw [thick, blue] (2.3,2) arc (0:180: .3) ;



\end{tikzpicture}
\caption {\label 2}
\end{figure}

An innermost returning arc can be removed as shown in Figure 4 by an isotopy of the simple closed curve, but a returning arc may be created in the other $2$-simplex containing the edge.

\

Notice that the number of intersections with the $1$-skeleton is reduced by two.  The process of  removing returning arcs must end with a track or with a loop inside a $2$-simplex.
The set of vertices inside the curve does not change and so it will end with a track if initially there were vertices both inside and outside the curve.

 \begin{figure}[htbp]
\centering
\begin{tikzpicture}[scale =.8]

 \draw  (0,2)--(2,4)--(4,2)--(2,0)--(0,2)--(4,2) ;
\draw (2.5,2) --(3,1) ;
\draw (1.5,2) --(1,1) ;

\draw (2.5,2) arc (0:180: .5) ;
\draw [red] (1.05, .95) --(1.5,1.8)--(2.5,1.8)--(2.95,.95) ;
  \draw  (0,2) -- (4,2) ;

 \draw  (6,2)--(8,4)--(10,2)--(8,0)--(6,2)--(10,2) ;
\draw (8.5,2) --(7.6,.4) ;
\draw (7.5,2) --(7,1) ;

\draw (8.5,2) arc (0:180: .5) ;
\draw [red] (7.05, .95) --(7.5,1.8)--(8.3,1.8)--(7.5,.45) ;
  \draw  (6,2) -- (10,2) ;


\end{tikzpicture}
\caption {A Returning Arc}
\end{figure}

  A pattern $P$ in $K$ is a set of tracks that are simple closed curves and each  track separates so the tracks are the edges of tree
$D_P$ in which the vertices are the connected components after removing the tracks.

Suppose $T$ is a simple closed curve in $K$ that is in general position with respect to the triangulation.   Then $T$ is a track (or rather it can be straightened to be a track) if and only if in the intersection with each $2$-simplex there are no returning arcs.

Let $K$ be a $2$-complex that is a triangulation of a $2$-sphere.  Let $P$ be a pattern in $K$ consisting of a maximal set of tracks in $K$ no two of which are parallel.  Such a maximal set exists since the number of non-parallel tracks in a finite $2$-complex is bounded by the number of vertices plus the number of $2$-simplexes.  If the tracks in a pattern separate as they do 
in this situation, or more generally if $H^1(K, \Z _2) =0$, then any pattern $P$ in $K^2$ will determine
a  tree $D_P$ in which the edge set is the set of tracks and the vertex set is the set of components of $K^2 - P$.

 \begin{figure}[htbp]
\centering
\begin{tikzpicture}[scale=.8]

\fill [pink] (6,1.8)--(8,0)--(8.4,0)--(10,1.4)--(10,1.8) --(6,2.2) --cycle;
\fill [pink] (0,1.8)--(4,2.2)--(4,2.6)--(2.4,4)--(2,4) --(0,2.2) --cycle;
\fill [pink] (2,0)--(2.4,0) --(4,1.4)--(4,1.8)--cycle;
\fill [pink] (8,4)--(8.4,4) --(10,2.6)--(10,2.2)--cycle;

  \draw  (0,0) -- (0,4) ;
  \draw  (4,0) -- (4,4)--(0,4) ;
   \draw  (0,0) -- (4,4) ;
  \draw  (4,0)--(0,0) ;
  
\draw   (0,0) node {$\bullet $} ;
\draw   (4,0) node {$\bullet $} ;	
\draw  (0,4) node {$\bullet $} ;
\draw  (4,4) node {$\bullet $} ;
\draw   (16,0) node {$\bullet $} ;
\draw   (12,0) node {$\bullet $} ;	
\draw  (12,4) node {$\bullet $} ;
\draw  (16,4) node {$\bullet $} ;

 \draw  (6,0)--(6,4) --(10,0)--(10,4) ; 

\draw   (6,0) node {$\bullet $} ;
\draw   (10,0) node {$\bullet $} ;
\draw  (6,4) node {$\bullet $} ;
\draw  (10,4) node {$\bullet $} ;


\draw [blue] (0,2.6)--(1.6, 4);
\draw [blue] (4,2.6)--(2.4, 4);

\draw [blue] (6,2.6)--(7.6, 4);
\draw [blue] (10,2.6)--(8.4, 4);

\draw [blue] (0,1.4)--(1.6, 0);
\draw [blue] (4,1.4)--(2.4, 0);

\draw [blue] (6,1.4)--(7.6, 0);
\draw [blue] (10,1.4 )--(8.4, 0);

\draw [red,thick] (6,1.8)--(8,0) ;

\draw [red,thick] (8,4)--(10,2.2) ;

\draw [red,thick] (6,2.2)--(10,1.8) ;

\draw [red,thick] (0,2.2) -- (2, 4) ; 
\draw  [red,thick] (2, 0) -- (4, 1.8) ; 


\draw [red,thick] (0,1.8)--(4,2.2) ;

\draw (6,0) --(10,0) ;

\draw (6,4) --(10,4) ;

\draw  [left] (0,0) node {$u$} ;
\draw   [left] (6,0) node {$u$} ;
\draw  [left] (0,4) node {$v$} ;
\draw  [left] (6,4) node {$v$} ;



\draw  [right ] (4,4) node {$z$} ;
\draw  [right ] (4,0) node {$w$} ;

\draw  [right ] (10,4) node {$z$} ;
\draw  [right ] (10,0) node {$w$} ;

\draw [red] (13, 2)--(15,2);
\draw [blue] (12,4)--(13,2)--(12,0) ;
\draw [blue] (16,4)--(15,2)--(16,0) ;
\draw [red] (15,2) node {$\bullet $} ;

\end{tikzpicture}

\caption {\label 2}
\end{figure}

An example is shown in Figure 2 for the tetrahedron $T$.  The maximal pattern $P$ consists of four blue $3$-tracks and one red $8$-track.   The tree $D_P$ will have five edges corresponding to the tracks of $P$ and six vertices,  two of degree $3$ and four of degree $1$, as shown.   The component corresponding to the right hand vertex of degree $3$ is shown shaded.

\begin{theo}  Every vertex in the tree $D_P$ has degree (valency) one  or three.  A component of degree one contains one vertex.   A component of degree three contains no vertex and is a disc with two smaller discs removed.

\end {theo}

\begin{figure}[htbp]

\begin{tikzpicture}[scale=1.6]

 \centering


\fill [pink]  (12.75,1)--(12.75, 1.25) --(13.8,1.25)--(13.8,2)--(13.2,2)--(13.2, 2.2)--(14.8,2.2)--(14.8,2)--(14.2,2)--(14.2,1.25)--(15.25,1.25)--(15.25,1)--cycle;
 
\fill [pink]  (13.2, 2.2) --(13.4,2.2)--(13.4, 3)--(14.6,3)--(14.6,2.2)--(14.8,2.2)--(14.8,3.2)--(13.2,3.2)--cycle ;
\fill [pink] (12.75,1)--(13,1)--(13,-.3)--(15,-.3)--(15,1)--(15.25,1)--(15.25,-.7)--(12.75,-.7)--cycle;

\draw [red](12.75,1.25)--(13.8,1.25)--(13.8,2) --(13.2,2)--(13.2, 3.2)--(14.8,3.2)--(14.8,2)--(14.2,2)--(14.2,1.25)--(15.25,1.25) ;
\draw [red] (13.4,2.2)--(13.4, 3)--(14.6,3)--(14.6,2.2)--cycle;
\draw [red] (13,1)--(15,1)--(15,-.3)--(13,-.3)--cycle;;
\draw (13.7,1.3)--(14.3,1.3)--(14,.9)--cycle;
\draw (13.7,1.9)--(14.3,1.9)--(14,2.3)--cycle;

\draw [red] (12.75,1.25) --(12.75, -.7)--(15.25,-.7) --(15.25,1.25);
\draw [red] (13.8,1.6)--(14.2,1.6);
\draw  [left] (13.8,1.6) node {$_a$} ;
\draw  [right] (14.2,1.6) node {$_b$} ;
\draw (14, -1.1) node {(i)} ;
\draw [left] (13.2,2.6) node {$_u$};
\draw [right] (13.4,2.6) node {$_s$};
\draw [right] (13, .2) node {$_t$};
\draw [right] (14.8, 2.6) node {$_u$};
\draw [left] (12.75,.2) node {$_u$};

\end{tikzpicture}
\end {figure}

\

 \begin{figure}[htbp]
\centering
\begin{tikzpicture}[scale=.5]
\draw (0,0)--(0,8)--(8,8)--(8,0)--cycle ;
\draw (1,1)--(3.5,1)--(3.5,7)--(1,7)--cycle;
\draw (7,1)--(4.5,1)--(4.5,7)--(7,7)--cycle;
\fill [pink] (0,0)--(0,8)--(1,8)--(1,0)--cycle;
\fill [pink] (8,0)--(8,8)--(7,8)--(7,0)--cycle;
\fill [pink] (1,0)--(7,0)--(7,1)--(1,1)--cycle;
\fill [pink] (1,7)--(7,7)--(7,8)--(1,8)--cycle;
\fill [pink] (3.5,1)--(3.5,7)--(4.5,7)--(4.5,1)--cycle;

\draw  [left ] (3.5,4) node {$_a$} ;
\draw  [right] (4.5,4) node {$_b$} ;
\draw [red] (3.5,4)--(4.5,4) ;
\draw (1.2,4) node {$_s$};
\draw  (6.8,4) node {$_t$};
\draw  (2,8.2) node {$_u$};
\draw (4,-1) node {(ii)};
\draw (2.5,6.5)--(4,8.2)--(5.5, 6.5)--cycle ;
\draw (2.5,1.5)--(4,-.2)--(5.5, 1.5)--cycle ;

\end{tikzpicture}

\caption {\label 2}
\end{figure}

\begin {proof}
Suppose $P$ intersects a $1$-simplex 
$\gamma $   in more than one point.  Let $a,b$ be adjacent points  of $\gamma\cap P$.  Removing small open neighbourhoods of  $a,b$ and replacing them with lines
parallel to $\gamma $ joining the end points of the open neighbourhoods will create either one simple closed curve (scc)   or two simple closed curves.     If the points $a,b$ lie in the same track $u$,  then the process of removing small neighbourhoods of  $a,b$ and joining the ends with lines parallel to $\gamma$ will create two simple closed curves $S,T$.  This can be seen from Figure 6(i). Now removing any returning arcs from $S,T$ will create two tracks,
which must be parallel to tracks $s,t$ in $P$. The removal of returning arcs cannot end with the total removal of an  scc.  If this happens then the last returning arc removed must have been a returning arc in the original track $u$, and so it would not have been a track.  Both $s$ and $t$ have fewer intersections with the $1$-skeleton than $u$. The region to which the interval $ab$ belongs is bounded by the tracks $s,t,u$, and so has degree three.

If $a,b$ lie in distinct tracks $s,t$ then just one scc $U$ is created, as in Figure 6(ii).  Now removing returning arcs will end in a track parallel to a track $u$ in $P$, and again the region of the interval  $ab$ has index three.  The process of removing returning arcs must eventually stop as the number of intersection points with the $1$-skeleton is reduced by two at each removal.
In the end  $U$ has become a track.   It cannot end with  $U$ being removed entirely  as this would mean that $s,t$ were parallel.
 Since we are dealing with a maximal set of tracks, this track  will be parallel to a  track $u$ in $P$.   It can be seen that 
there is a vertex $v$ of $D_P$ with incident edges $s,t,u$.   The region corresponding to $v$ will consist of  triangular regions in two $2$-simplexes joined by three bands
The vertex component is a disc with two discs removed.   This can be described as a disc with two punctures.

   Removing a returning arc joining two points $a,b$  in one $2$-simplex will create a returning arc in the other $2$-simplex containing $a,b$ if $a,b$
are joined to points in the same edge.  Thus  it will be a track if and only  the situation in the  two $2$-simplexes containing $\gamma$ is as in Figure 7(i).  If the situation is as Figure 7(ii) then  one returning arc is created, and if it is as in Figure 7(iii) then two returning arcs are created.

 \begin{figure}[htbp]
\centering
\begin{tikzpicture}[scale=.6]

 \draw  (1,3)--(1.7,2) --(1,1);
\draw (3,3)--(2.3,2) --(3,1) ;
\draw (7,3)--(7.7,2)--(7,1);
\draw (7.35,  3.35)--(8.3,2)--(9,1);
\draw (13,3)--(13.7,2)--(13,1);
\draw (13.35,  3.35)--(14.3,2)--(13.35,.65);
\draw [blue] (1.05,3.05) --(1.605,2.25)--(2.395,2.25)--(2.95,3.05) ;
\draw [blue] (1.05,.95) --(1.605,1.75)--(2.395, 1.75)--(2.95,.95) ;
\draw [blue] (7.05,.95) --(7.605,1.75)--(8.395, 1.75)--(8.95,.95) ;

 \draw  (6,2)--(8,4) --(10,2)--(8,0)--cycle ; 
 \draw  (0,2)--(2,4) --(4,2)--(2,0)--cycle ; 
 \draw  (12,2)--(14,4) --(16,2)--(14,0)--cycle ; 
\draw [red] (1.6,2.2)--(2.4,2.2) ;
\draw [red] (1.6,1.8)--(2.4,1.8) ;
\fill [pink] (1.6,2.2)--(2.4,2.2)--(2.3,2)--(2.4,1.8)--(1.6,1.8)--(1.7,2)--cycle ;
\fill [pink] (7,3)--(7.35,3.35)--(8.3,2)--(8.4,1.8)--(7.6,1.8)--(7.7,2)--cycle ;
\fill [pink] (13,3)--(13.35,3.35)--(14.3,2)--(13.35,.65)--(13,1)--(13.7,2)--cycle ;



  \draw  (0,2) -- (4,2) ;
  \draw  (6,2) -- (10,2) ;
   \draw  (12,2) -- (16,2) ;

\draw  [above,left ] (1.6,2.1) node {$_a$} ;
\draw  [above,right] (2.3,2.1) node {$_b$} ;
\draw  [right] (1.1,2.8) node {$_s$} ;
\draw  [left] (2.9,2.8) node {$_t$} ;

\end{tikzpicture}

\end{figure}

 \begin{figure}[htbp]
\centering
\begin{tikzpicture}[scale=.6]


\draw [blue] (1.05,3.05) --(1.605,2.25)--(2.395,2.25)--(2.95,3.05) ;
\draw [blue] (1.05,.95) --(1.605,1.75)--(2.395, 1.75)--(2.95,.95) ;
\draw [blue] (7.05,.95) --(7.605,1.75)--(8.395, 1.75)--(8.95,.95) ;

 \draw  (6,2)--(8,4) --(10,2)--(8,0)--cycle ; 
 \draw  (0,2)--(2,4) --(4,2)--(2,0)--cycle ; 
 \draw  (12,2)--(14,4) --(16,2)--(14,0)--cycle ; 
\fill [pink] (7,3)--(7.35,3.35)--(8.1,2.25)--(7.5,2.25)--cycle ;
\fill [pink] (13,3)--(13.35,3.35)--(14.1,2.25)--(13.5,2.25)--cycle ;
\draw (13,3)--(13.35,3.35)--(14.1,2.25)--(13.5,2.25)--cycle ;
\draw (7,3)--(7.35,3.35)--(8.1,2.25)--(7.5,2.25)--cycle ;
\fill [pink] (13,1)--(13.35,.65)--(14.1,1.75)--(13.5,1.75)--cycle ;
\draw (13,1)--(13.35,.65)--(14.1,1.75)--(13.5,1.75)--cycle ;



  \draw  (0,2) -- (4,2) ;
  \draw  (6,2) -- (10,2) ;
   \draw  (12,2) -- (16,2) ;


\draw  (2,-.5) node {(i)} ;
\draw  (8,-.5) node {(ii)} ;
\draw  (14,-.5) node {(iii)} ;

\end{tikzpicture}

\caption {\label 2}
\end{figure}


 If a vertex region corresponds to a vertex of $D_P$  of degree one, then it will be bounded by a track intersecting each edge at most once.  The region must contain a single vertex of the triangulation.   Conversely every vertex of the triangulation determines a track surrounding it,  and the region inside this track will correspond to a vertex of $D_P$ of degree one.
 A $1$-simplex of $K$ determines a path in $D_P$ starting and ending in a vertex of degree one. Every other vertex visited has degree three.  The path will backtrack at a vertex if and only if the corresponding part of the $1$-simplex joins two points in the same track.

\end {proof}

Let $v$ be the number of vertices of the triangulation, $e$ the number of edges ($1$-simplexes) and $f$ the number of faces ($2$-simplexes), then  $3f =2e$ since
every face has $3$ edges and each edge belongs to two faces.  Thus $f$ is even. Every track determines a proper decomposition of the set of vertices, i.e. there are no tracks bounding a region with no vertices.   Hence the
 tree $D_P$ has $v$ vertices of degree one.  If $v_P$ is the number of vertices of $D_P$, there are $v_P -v$ vertices of degree $3$.  But our argument above has shown that every vertex of degree $3$ involves exactly two faces.  Hence $2v_P-2v = f$.  The number of vertices of a finite tree is the number of edges plus one.  Hence $v_P = e_P +1$.   Thus $e_P = v_P-1 = {1\over2} f +v-1$, which shows the number of  tracks in a maximal pattern does not vary with the pattern.   
 
 For example if $K=T$ the faces of a tetrahedron, so that
 $v=f=4$, then $e_P =5$.   There are four $3$-tracks and one other.
 There are infinitely many possibilities for the track that is not a $3$-track. There is the $4$-track and also the ones  shown in Figure 9.  Here $m$ and $n$  are positive integers
 which are  the number of parallel lines in the places indicated.    This will be a $4(m+n)$-pattern, which will be a track if $m$ and $n$ are coprime.  It will be a pattern since the numbers 
 match up at every edge and we have shown that the only  possibility for a pattern which does not include $3$-tracks are parallel copies of the same track.  There are parallel copies if and only if $m$ and $n$ are not coprime.   These are the only tracks in $T$ that are not $3$-tracks or $4$-tracks.
 
 For $m=n=1$ we have an $8$-track which is important in our theory.

 \begin{figure}[htbp]
\centering
\begin{tikzpicture}[scale=.7]

  \draw  (0,0) -- (0,4) ;
  \draw  (4,0) -- (4,4)--(0,4) ;
   \draw  (0,0) -- (4,4) ;
  \draw  (4,0)--(0,0) ;
  
\draw   (0,0) node {$\bullet $} ;
\draw   (4,0) node {$\bullet $} ;	
\draw  (0,4) node {$\bullet $} ;
\draw  (4,4) node {$\bullet $} ;

 \draw  (6,0)--(6,4) --(10,0)--(10,4) ; 

\draw   (6,0) node {$\bullet $} ;
\draw   (10,0) node {$\bullet $} ;
\draw  (6,4) node {$\bullet $} ;
\draw  (10,4) node {$\bullet $} ;


\draw [red] (6,1.8)--(8,0) ;

\draw [red] (8,4)--(10,2.2) ;

\draw [red] (6,2.2)--(10,1.8) ;

\draw [red] (0,2.2) -- (2, 4) ; 
\draw  [red] (2, 0) -- (4, 1.8) ; 


\draw [red] (0,1.8)--(4,2.2) ;

\draw (6,0) --(10,0) ;

\draw (6,4) --(10,4) ;

\draw  [red ] (1,3.1) node {$m$} ;
\draw  [red ] (3,.9) node {$m$} ;

\draw  [red ] (9,3.1) node {$m$} ;
\draw  [red ] (7,.9) node {$m$} ;

\draw  [red ] (1,1.9) node {$n$} ;
\draw  [red ] (3,2.1) node {$n$} ;

\draw  [red ] (9,1.9) node {$n$} ;
\draw  [red ] (7,2.1) node {$n$} ;

\draw  [left] (0,0) node {$u$} ;
\draw  [left] (6,0) node {$u$} ;
\draw  [left] (0,4) node {$v$} ;
\draw  [left] (6,4) node {$v$} ;



\draw  [right ] (4,4) node {$z$} ;
\draw  [right ] (4,0) node {$w$} ;

\draw  [right ] (10,4) node {$z$} ;
\draw  [right ] (10,0) node {$w$} ;

\end{tikzpicture}
\caption {\label 2}

\end{figure}

\end {section}
\begin {section} {Patterns in $3$-manifolds}
 
 Let $M$ be a $3$-manifold.
Let $f : S^2\rightarrow M$  be an injective  general position map (see Hempel \cite {[JH76]}, Chapter 1), in which $f$ is in general position with respect to a triangulation $K$ of $M$.   An $i$-piece of $f$ is defined to be a component of $f^{-1}(\sigma)$ where $\sigma$ is an $(i+1)$-simplex of $K$. Thus a  $0$-piece is a point of $S^2$. A $1$-piece is either an scc (simple closed curve) or an arc joining two $0$-pieces.  If there are no $1$-pieces that are scc's, then each $2$-piece  has boundary that is is a union of $1$-pieces.
 One can use surgery along simple closed curves to change $f$ to a map in which there are no $1$-pieces that are scc's,  and in which every $2$-piece is a disc.
The $2$-pieces will then give  a cell decomposition (tessellation) of the $2$-sphere.   If $R$ is a $1$-piece with end points $u,v$ whose images under $f$ are in the same $1$-simplex, then,  the restriction of $f$ to $R$ is called a returning arc.  This is a similar definition to the one for $2$-simplexes.

If $M$ is a $3$-manifold and $M$ is triangulated so that $M=|K|$ where $K$ is a finite $3$-complex,   then a pattern $P$ in $|K^2|$ determines a {\it patterned surface} $S$   such that  for each $3$-simplex $\rho $, $S\cap |\rho |$ consists of 
disjoint properly embedded discs and  $S\cap |K^2| =P$.  A  patterned surface is determined, up to isotopy, by the intersection $P\cap |K^1|$.  If the pattern in $|K^2|$ is normal,
then the patterned surface is a {\it normal surface}.

If $f :S^2\rightarrow M$ has no $1$-pieces that are returning arcs or scc's, then the intersection of $f(S^2)$ with the $2$-skeleton of $M$ is a pattern $P$.
   There is an isotopy from $f$ to  $f' : S^2\rightarrow M$ in which the image is the patterned surface determined by $P$.

Let $\gamma $ be a $1$-simplex of $M$.  Two points $p,q \in \gamma \cap f(S^2)$ are said to be {\it }removable if there is a homotopy from  $f$  to a map $f' : S^2 \rightarrow M$ such that $f(x)=f'(x)$ for every 
$x$ that is not in the interior of a simplex with $ \gamma $ as a face and $\gamma \cap f'(S^2)$ is the same as  $\gamma \cap f(S^2)$ but with $p,q$ removed.

The pair of end points of a returning arc $R$ are removable by the following homotopy.
 Let $\sigma $ be the $2$-simplex of $K$ such that $f(R) \subset \sigma $.  Let $V$ be a regular neighbourhood of $R$ in $S^2$.  Let $ V^{\circ }$ be the interior of $V$
regarded as a subspace of $V$.  Let $\beta V$ be the boundary of $V$ regarded as a subspace of $S^2$, so that $\beta V = V- V^{\circ}$.
Let  $\gamma$ be the $1$-simplex containing the end points of $R$.  The regular neighbourhood $V$ is a disc and $\beta V = \delta V$ is a  simple closed curve in $S^2$.  The union of all the $3$-simplexes that contain $\gamma $ is a closed ball $B$ and $f(\beta V) \subset B^{\circ }-\sigma $, which is contractible.  Define $f': S^2 \rightarrow M$ so that $f'$ is continuous, $f'$ and $f$ are the same when restricted to $S^2 - V^{\circ}$, and $f'(V) \subset B^{\circ} - \sigma$.  Note that removing $p$ may create more $1$-pieces that are returning arcs or sccs, but the size of the intersection with the $1$-skeleton goes down by two.   

Let $M=|K|$ be a $3$-manifold, where $K$ is a finite $3$-complex.   Let $H^1(K,\Z_2)=0$ so that tracks in $K^2$  separate.   In this situation any pattern $P$ in $K^2$ will determine
a finite tree $D_P$ in which the edge set is the set of tracks and the vertex set is set of components of $K^2 - P$.  In the situation when the patterned surface corresponding to $P$ is a maximal set of normal $2$-spheres, it is proved in \cite {[T94]} that a vertex of $D_P$ of degree (valency)  greater than one corresponds to a component  which is  a punctured 
$3$-ball.  In fact a slightly stronger result is true.
\begin {theo} Let $P$  be a pattern in $K^2$, in which no two tracks are parallel, for which the patterned  surface corresponding to $P$ is a maximal set of  $2$-spheres,  then in the tree $D_P$ every vertex has degree one, two or three.   Each vertex of degree two or three is  a component of $K^2 - P$ that is a punctured $3$-ball.  The same is true if  the $2$-spheres are restricted to normal spheres.
\end {theo}
\begin {proof} Part of the previous proof works with some adjustments.

\begin{figure}[htbp]
\centering
\begin{tikzpicture}
\fill [pink]  (4,0) --(3,0)--(3.5,.5)--(4,0);
\fill [pink] (0,0)--(.45,.45)--(.9,0) --(0,0);
\fill [pink]  (2,2)--(2.5,1.5)--(1.5,1.5)--(2,2);
\fill [pink] (2.7,1.3)--(3,1) --(2.7,0)--(2.2,0)--(2.7,1.3); ;
\fill [pink]  (1,1)--(1.6,0)--(1.2,0)--(.65,.65)--(1,1););

\fill [pink]  (10,0) --(9.2,0)--(9.4,.6)--(10,0);
\fill [pink] (6,0)--(6.45,.45)--(6.9,0) --(6,0);
\fill [pink]  (8,2)--(8.5,1.5)--(7.5,1.5)--(8,2);
\fill [pink]  (7,1)--(7.6,0)--(7.2,0)--(6.65,.65)--(7,1);
\fill [pink] (2.2,0)--(2.7,0)--(3,1)--(2.7,1.3)--(1.4,1.4) -- (1.2,1.2)--(2.2,0);
\fill[pink](8.2,0)--(8.7,0)--(9,1)--(8.7,1.3)--(8.2,0);
\draw [red]  (4,0) --(3,0)--(3.5,.5)--(4,0);
\draw[red] (0,0)--(.45,.45)--(.9,0) --(0,0);
\draw[red]  (2,2)--(2.5,1.5)--(1.5,1.5)--(2,2);
\draw[red] (2.7,1.3)--(3,1) --(2.7,0)--(2.2,0)--(1.2,1.2)--(1.4,1.4)--(2.7,1.3); 
\draw[red]  (1,1)--(1.6,0)--(1.2,0)--(.65,.65)--(1,1);

\draw[red]  (10,0) --(9.2,0)--(9.4,.6)--(10,0);
\draw[red](6,0)--(6.45,.45)--(6.9,0) --(6,0);
\draw[red]  (8,2)--(8.5,1.5)--(7.5,1.5)--(8,2);
\draw[red]  (7,1)--(7.6,0)--(7.2,0)--(6.65,.65)--(7,1);
\draw[red](8.2,0)--(8.7,0)--(9,1)--(8.7,1.3)--(8.2,0);

\draw (0,0) -- (4,0)--(2,2) -- (0,0) ;
\draw (6,0)--(10,0)--(8,2)--(6,0) ;



\draw [red] (1.5,1.5)--(2.5,1.5) ;

\draw [red] (.65,.65)--(1.2,0) ;
\draw [red] (3,0)--(3.5,.5);
\draw [red] (.9,0)--(.45,.45);
\draw [red] (2.7,0)--(3,1);

\draw [red] (9.4,.6)--(9.2,0);
\draw [red] (6.9,0)--(6.45,.45);


\end{tikzpicture}

\caption {\label 2}

\end{figure}

Let $V$ be the closure  of a component of $K^2-P$ corresponding to a vertex of $D_P$ of degree at least two.  An argument from \cite {[T94]} shows that there will be a
  $1$-simplex of $K$, intersecting  two distinct tracks $s,t$ of $P$ in the boundary of $V$ at adjacent points $a,b$ on the $1$-simplex $\gamma $ where $a\in s\cap \gamma,  b\in t\cap\gamma$.    Suppose this does not happen.       Then if we cut $V$ up along the faces of every $2$-simplex that it intersects (or perhaps more precisely removing a small open neighbourhood of those faces )  we will be cutting along 
  discs,  where each disc is a shaded region in Figure 9.   We end up with $3$-balls that are the components of intersection of $V$ with a particular $3$-simplex.  We are assuming that adjacent points  of intersection of $V$ with the $1$-skeleton belong to the same track.  Cutting along such discs will disconnect $V'$, but there will be a  component with more than one boundary component.   Cutting along the boundary of all  the $3$-simplexes yields $3$-balls with connected boundary $2$-spheres, each intersecting at most one track.  This  is a contradiction.

 \begin{figure}[htbp]
\centering
\begin{tikzpicture}[scale=.6]


  \draw  (0,0) -- (0,4) ;
  \draw  (4,0) -- (4,4)--(0,4) ;
   \draw  (0,4) -- (4,0) ;
  \draw  (4,0)--(0,0) ;

 \draw  (6,0)--(6,4) --(10,0)--(10,4) ;

\draw [red]  (1.5,0) -- (0,2) ;
\draw [red]  (7.5,0) -- (6,1.8) ;
\fill [pink] (7.5,0)--(6,1.8)--(6,3)--(8.5,0)--cycle ;

\draw [red]  (2.5,0) -- (4,2) ;
\draw [red]  (8.5,0) -- (6,3) ;

\draw (6,0) --(10,0) ;

\draw (6,4) --(10,4) ;


  \draw  (12,0) -- (12,4) ;
  \draw  (16,0) -- (16,4)--(12,4) ;
    \draw  (16,0)--(12,0) ;

\fill [pink] (15,0) --(15,1)--(15.2,1.2)--(14,2.4)--(14,0)--cycle ;

\draw [blue] (14.1,4)--(14.1,2.4)--(15.2,1.3)--(16,2.1);
\draw [red]  (14,0) --(14,4) ;
\draw [red] (15,0) --(15,1)--(16,2);
\fill [pink] (20,0)--(20,4)--(21,4)--(21,0)--cycle ;
\draw [red]  (20,0) --(20,4) ;
\draw [red] (21,0)--(21,4);
\fill [pink] (1.5,0)--(1,0.7)--(3,0.7)--(2.5,0)--cycle ;
\draw [blue] (0,2.1)--(1,0.75)--(3,0.75)--(4,2.1) ;


\draw (18,0) --(22,0) ;

\draw (18,4) --(22,4) ;

 \draw  (12,4) -- (16,0) ;

\draw  [left] (0,0) node {$u$} ;
\draw  [left] (0,4) node {$v$} ;

\draw  [right ] (4,0) node {$w$} ;
\draw  [right ] (4,4) node {$z$} ;


\draw (1.2,0.15) node {$_a$} ;
\draw   (2.8,0.15) node {$_b$} ;
\draw  [right] (.5,1.5)  node {$_s$} ;
\draw  [left] (3.5,1.5) node {$_t$} ;

 \draw  (18,0)--(18,4) --(22,0)--(22,4) ;

\end{tikzpicture}

\end{figure}

 \begin{figure}[htbp]
\centering
\begin{tikzpicture}[scale=.6]


  \draw  (0,0) -- (0,4) ;
  \draw  (4,0) -- (4,4)--(0,4) ;
   \draw  (4,4) -- (0,0) ;
  \draw  (4,0)--(0,0) ;
  
\

\draw [red]  (1.5,0) -- (0,2) ;
\draw [red]  (7.5,0) -- (6,1.8) ;
\fill [pink] (7.5,0)--(6,1.8)--(6,3)--(8.5,0)--cycle ;

\draw [red]  (2.5,0) -- (4,2) ;
\draw [red]  (8.5,0) -- (6,3) ;

\draw (6,0) --(10,0) ;
 \draw  (6,4)--(6,0) --(10,4)--(10,0) ; 

\draw (6,4) --(10,4) ;


  \draw  (12,0) -- (12,4) ;
  \draw  (16,0) -- (16,4)--(12,4) ;
    \draw  (16,0)--(12,0) ;

\fill [pink] (15,0) --(15,1.9)--(14,1.9)--(14,0)--cycle ;

\draw [blue] (14.1,4)--(14.1,1.9)--(16,1.9);
\draw  (6,4)--(6,0) --(10,4)--(10,0) ; )--(16,2.1);
\draw [red]  (14,0) --(14,4) ;
\draw [red] (15,0) --(15,1.8)--(16,1.8);
\fill [pink] (20,0)--(20,4)--(21,4)--(21,0)--cycle ;
\draw [red]  (20,0) --(20,4) ;
\draw [red] (21,0)--(21,4);
\fill [pink] (1.5,0)--(1,0.7)--(3,0.7)--(2.5,0)--cycle ;
\draw [blue] (0,2.1)--(1,0.75)--(3,0.75)--(4,2.1) ;


\draw (18,0) --(22,0) ;

\draw (18,4) --(22,4) ;

 \draw  (16,4) -- (12,0) ;

\draw  [left] (0,0) node {$u$} ;
\draw  [left] (0,4) node {$v$} ;

\draw  [right ] (4,0) node {$w$} ;
\draw  [right ] (4,4) node {$z$} ;

\draw (2,-1) node {(i)} ;
\draw  (8,-1)  node {(ii)} ;
\draw  (14,-1)  node {(iii)};
 \draw  (6,4)--(6,0) --(10,4)--(10,0)  ;
\draw   (20,-1) node {(iv)};

 \draw  (18,4)--(18,0) --(22,4)--(22,0) ;

\end{tikzpicture}
\caption {\label 2}
\end {figure}

Assume then that $a,b$ are as above in different tracks $s,t$. 
There will now be a $2$-sphere  $S$ which is the union  of the patterned surface corresponding to $s$ with a small neighbourhood of $a$ removed and the  patterned surface corresponding to $t$ with small neighbourhoods of $b$ removed   together with a tube joining the boundary circles.
The $2$-sphere $S$ will usually not be a patterned $2$-sphere,  but it will become one by removing returning arcs, as described above. 
It you have a surface in general position in a $3$-manifold, then one can  get a patterned surface by removing any simple closed curves in $2$-simplexes
and all returning arcs, unless the whole manifold is obtained.


    Now removing a returning arc in $S$  may create more than one  returning arc, as in Figure 11, since $\gamma$ will belong to more than two $2$-simplexes. The initial
$2$-sphere $S$ bounds a $3$-ball consisting of two $3$-balls bounded by $s$ and $t$ joined together by the  filled tube. In $M$ $s$ and $t$ might not bound $3$-balls, but we assume they do for this argument.   Removing a returning arc is done by expanding this $3$-ball,  and the possible expansions 
are shown in Figure 10.  In Figure 10(i) no returning arc has been created.  In Figure  10(ii)  returning arcs are created in simplexes with face $uv$ or $uz$ and the $3$-ball is expanded to include a $3$-ball bounded by an annulus and two discs each  bounded by a $3$-track.  In Figure 10(iii) returning arcs may  be created in simplexes with faces $uz$ or  $vw$ and the $3$-ball is expanded to include a $3$-ball bounded by the shaded region and a disc bounded by the $3$-track shown in blue.  In Figure 10(iv) returning arcs may be  created in simplexes with a face in four of the six edges of $T$ and the $3$-ball is expanded to include a $3$-ball bounded by an annulus and two discs each  bounded by a $4$-track. Note that in the process of expanding the $3$-ball in the cases of Figure 10(ii) and Figure 10(iv) as just described, two returning arcs will occur in the same $2$-simplex but $s, t$ are not parallel.  The two returning arcs are both removed and replaced by two parallel lines which are part of parallel tracks in the tetrahedron.

Once all the returning arcs have been removed the boundary of the expanded $3$-ball will be a patterned $2$-sphere $U$ unless we have removed all of $s$  and $t$  which  would mean that $V$ had just two boundary $2$-spheres.   If $V$ had more than two boundary $2$-spheres, then since we had a maximal set of tracks corresponding to patterned $2$-spheres, $U$ must be parallel to $u$, a track in $P$, and $V$ has $s,t,u$ as boundary $2$-spheres.  After all returning arcs have been removed one ends up with a track that determines a  patterned $2$-sphere $U$ if the vertex region has degree more than $2$, and it has degree $2$ if all of $s$ and $t$ have been removed.
 This can happen 
without $s$ and $t$ being parallel, unlike the earlier theorem.

Each $2$-piece of $u$ will be parallel to a $2$-piece of $s$ or $t$ apart from ones that arise from the Figure 10(i) or the Figure 10(iii) situation, which are $4$-sided and $3$-sided respectively.  Thus if $s$ and $t$ are normal patterns then so is $u$.  If we start with a pattern $P$ in which the tracks are a maximal set of non-parallel, normal $2$-spheres,  then the proof 
works fine. 
\end {proof}


\begin{figure}[htbp]
\centering
\begin{tikzpicture}[scale=.6]

  \draw  (5.5,.5)--(4.5,2)--(-.5,2)--(-1.5,.5);
  
  \draw [dashed] (.5,.5)--(-.5,2);
  \draw (-.5,2)--(-.5,4);
\draw (3.5,.5)--(4.5,2)--(4.5,4);

 \draw [red, dashed] (1.5,.5)-- (.5,2);
 \draw [red]  (-.5,.5)--(.5,2)--(.5,3) --(3.5,3)--(3.5,2)--(2.5,.5) ; 
\draw [red,dashed] (3.5,2)--(4.5, .5);


\draw [red] (1.8,3.1)--(2,3)--(1.8, 2.9) ;
\draw [red]  (-.35,1)--(0, 1.2)--(-.1, .85) ;
\draw [red]  (7.65,1)--(8, 1.2)--(7.9, .85) ;
\draw [red]  (10.7, 1)--(10.65, .7)--(11,.9) ;

  \draw  (13.5,.5)--(12.5,2)--(7.5,2)--(6.5,.5);
  
  \draw [dashed] (8.5,.5)--(7.5,2);
  \draw (7.5,2)--(7.5,4);
\draw (11.5,.5)--(12.5,2)--(12.5,4);

 \draw [red, dashed] (9.5,.5)-- (8.5,1.8)--(12, 1.8)--(13,.5);
 \draw [red]  (7.5,.5)--(8.4,1.7)--(11.4,1.7)--(10.5,.5) ; 
\draw [red,dashed] (3.5,2)--(4.5, .5);


\end{tikzpicture}
\caption {\label 2}
\end{figure}

\end {section}
\begin {section} {The Recognition Algorithm}


Let $M$ be a compact triangulated simply connected $3$-manifold with no boundary.     Let $P$ be a pattern in $K^2$ in which the tracks are a maximal collection of normal tracks which correspond to $2$-spheres in $M$ and no two are parallel.   These tracks are the edges of a finite tree $D_P$.  Then $M$ will be a $3$-sphere if and only if the regions corresponding to the vertices of
$D_P$ of degree one are all $3$-balls.   This is because all the other vertex regions are punctured $3$-balls and they will fit together to form one big  punctured $3$-ball.

A vertex region is a $3$-ball if it contains a vertex of the triangulation,  since, by the maximality of $P$, it will contain just one vertex.  Unlike the $2$-dimensional  case, there can be 
vertex regions of degree one that do not contain a vertex of the triangulation.

\begin {defi} 
A pattern $P$ is said to be almost normal if the intersection of $P$ with the boundary of every $3$-simplex  consists of $3$-tracks and $4$-tracks apart from exactly one $3$--simplex 
whose boundary  intersects $P$  in $3$-tracks and exactly one $8$-track.  A patterned surface is almost normal if it 
corresponds to an almost normal pattern.  
\end {defi}

\begin {theo} The manifold $M$ is a $3$-sphere if an only if every vertex region corresponding to a vertex of $D_P$ of degree one contains a vertex of $K$ or an almost normal $2$-sphere.
\end {theo}

This result is due to Hyam Rubinstein \cite {[R97]}. We are reworking the proof by Abigail Thompson \cite {[T94]}.

\begin {proof}

Let $M$ be a triangulated $3$-sphere..   Let $M_0$ be a component, obtained by cutting along the maximal collection of normal $2$-spheres, which has one boundary component  and which does not contain a vertex. Then $M$ is a $3$-sphere if and only if all such regions are $3$-balls.
The intersection of $M_0$ with a $2$-simplex will be the unshaded region  as in one of the two cases of Figure 9,  depending on whether the central region is in $M_0$.

There is isotopy between the boundary and the constant map.

 We can break the isotopy up into finitely many steps in which the intersection with the $1$-skeleton changes by a pair of points.



A  returning arc in a $2$-manifold is shown in Figure 4.  In a $3$-manifold there will be more than two $2$-simplexes with the same edge.  However one of these two pictures of Figure 4 will occur in each pair of $2$-simplexes one of which is  the  $2$-simplex containing the returning edge and the other $2$-simplex is  any one  of the other $2$-simplexes containing that edge.  It will be an isotopy of the surface.

 In fact every isotopy move  is the removal or addition of a {\it removable pair.}  We have seen the ends of a returning arc  are a removable pair.  Two adjacent  points  belonging to the same edge in an $n$-track in a $3$-simplex are also a removable pair.   Such pairs  will only exist for $n \geq 8$.   A pattern is a normal pattern if and only if it has no removable pairs.

In Figure 12 see how the removal of the two right hand points of an $8$-track gives two $3$-tracks in which the opposite pair of points are no longer in the same track.
Returning arcs will be created in at least one of the other  $2$-simplexes containing  the edge $wz$.

  \begin{figure}[htbp]
\centering
\begin{tikzpicture}[scale=.5]
  \draw  (0,0) -- (0,4) ;
  \draw  (4,0) -- (4,4)--(0,4) ;
   \draw  (0,0) -- (4,4) ;
  \draw  (4,0)--(0,0) ;

 \draw  (6,0)--(6,4) --(10,0)--(10,4) ;

\draw [blue ] (0,.5) --(.5,0) ;
\draw [red]  (6,2.2) --(10,1.8) ;

\draw [red] (6,1.8)--(8,0) ;
\draw [red] (6.8, .9)--(6.7,1.2)--(7, 1.1);
\draw [red] (8,4)--(10,2.2) ;

\draw [red] (9.1, 2.8)--(9,3.1)--(9.3, 3.05) ;
\draw [red] (1.1, 3.4)--(1,3.1)--(1.35,3.15);
\draw [red] (0,2.2) -- (2, 4) ; 
\draw  [red] (2, 0) -- (4, 1.8) ;

\draw [red] (0,1.8)--(4,2.2) ;
\draw  [red] (1.6,2.15)--(1.9, 2)--(1.6,1.8) ;
\draw  [red] (7.6,2.2)--(7.85, 2)--(7.6,1.85) ;
\draw (6,0) --(10,0) ;
\draw  [red] (3.1, 1.15)--(3,.9)--(3.3, .9) ;

\draw (6,4) --(10,4) ;
\draw  [left] (0,0) node {$_u$} ;
\draw  [left] (6,0) node {$_u$} ;
\draw  [left] (0,4) node {$_v$} ;
\draw  [left] (6,4) node {$_v$} ;


\draw  [right ] (4,0) node {$_w$} ;
\draw  [right ] (10,0) node {$_w$} ;
\draw  [right ] (4,4) node {$_z$} ;

\draw  [right ] (10,4) node {$_z$} ;

\draw [blue] (0,3)--(1,4);
\draw [blue] (0,3.5)--(.5,4);
\draw [blue] (0,2.5)--(1.5,4);
\draw [blue] (4,3)--(3,4);
\draw [blue] (10,3)--(9,4);
\draw [blue] (6,3)--(7,4);
\draw [blue] (6,3.5)--(6.5,4);
\draw [blue] (6,2.5)--(7.5,4);



  \draw  (15,0) -- (15,4) ;
  \draw  (19,0) -- (19,4)--(15,4) ;
   \draw  (15,0) -- (19,4) ;
  \draw  (19,0)--(15,0) ;

\draw [red] (16.1, 3.4)--(16,3.1)--(16.35,3.15);



\draw [red] (21,1.8)--(23,0) ;

\draw [red] (23,4)--(21,2.2) ;
\draw [red] (21.7, 3)--(22,3.1)--(21.9,2.8);

\draw [red] (15,2.2) -- (17, 4) ; 
\draw  [red] (15, 1.8) -- (17, 0) ; 

\draw [red] (15.7, 1)--(16,.9)--(15.9, 1.2);

\draw (21,0) --(25,0) ;
\draw (21,0) --(21,4) ;

\draw (21,4)--(25,0) --(25,4) ;

\draw [red] (21.8, .9)--(21.7,1.2)--(22, 1.1);

\draw (21,4) --(25,4) ;

\draw  [left] (15,0) node {$_u$} ;
\draw  [left] (21,0) node {$_u$} ;
\draw  [left] (15,4) node {$_v$} ;
\draw  [left] (21,4) node {$_v$} ;

\draw  [right ] (19,0) node {$_w$} ;
\draw  [right ] (25,0) node {$_w$} ;
\draw  [right ] (19,4) node {$_z$} ;

\draw  [right ] (25,4) node {$_z$} ;


\draw [blue] (0,1)--(1,0);
\draw [blue] (6,1)--(7,0);
\draw [blue] (6,.5)--(6.5,0);

\draw [blue] (15,3)--(16,4);
\draw [blue] (15,3.5)--(15.5,4);
\draw [blue] (15,2.5)--(16.5,4);
\draw [blue] (19,3)--(18,4);
\draw [blue] (25,3)--(24,4);

\draw [blue] (16,0)--(15,1) ;
\draw [blue] (22,0)--(21,1);
\draw [blue] (21.5,0)--(21,.5) ;
\draw [blue] (15.5,0)--(15,.5) ;
\draw [blue] (21,3)--(22,4);
\draw [blue] (21,3.5)--(21.5,4);
\draw [blue] (21,2.5)--(22.5,4);

\end{tikzpicture}

\caption {\label2}

\end{figure}

The situation in a $3$-sphere is different from that of a $2$-sphere in that there will be vertices of $D_P$ of degree one that do not contain a vertex of the triangulation.

 There will be an isotopy from the boundary to the constant map.

 Let $f :S^2 \rightarrow M$ be an injective map whose  image is the  normal $2$-sphere $\delta M_0$.
 Let $F : S^2  \times I \rightarrow M_0$
be a isootopy between $f $ and a constant map. For $t \in  I$ let $f_t : S^2 \rightarrow M_0,  f_t(s) = F(s,t)$.
 We can assume that for all but finitely many values $t, f_t$ meets the $1$-skeleton $T^1$
of the triangulation transversely and for each $t$ for which the map $f_t $ does not meet $T^1$
transversely, there is precisely one point where $f_t$ is tangential to $T^1$.
Let $t_1',t_2',\dots ,t_n'$
be the values of $t $ for which $f_t$ does not meet $T^1$
transversely and
put $t_0 = 0, t_{n+1}=1$. For $i = 1, . . . , n $ choose $t_i  \in (t_i',t_{i+1}')$
and put $f_i = f_{t_i}$.

Let $W_i$ be the set of intersections of $f_i(S^2)$
with the $1$-skeleton of $T$ and let $w_i = |W_i|$.
Rearrange the weights $w_i$
into a finite non-increasing sequence. Order these sequences
lexicographically. The width of $T$ is the minimum sequence of weights as $F$ ranges over
all possible isotopies.
If F realises the width of $T$, then $F$ is said to be in thin position. Now revert to the
original ordering of the $w_i$'s.
 Clearly for each $i$ either $w_{i+1} = w_i + 2$ or $w_{i+1} = w_i-2$. Note
that $w_0$   is the number of intersections of $\delta M_0$ with 
the $1$-skeleton. For $F$ in thin position, we are particularly interested in the values  $ i$ for which $w_{i-1}= w_{i+1} = w_i-2$. For
such an $i, f_i(S^2)$ is called a thick sphere, while if $w_{i-1} = w_{i+1} = w_i + 2$, then $f_i(S^2)$
is
called a thin sphere. Our interest will be in the first thick sphere.  Since there are no removable pairs in a normal surface, $w_1 = w_0 +2$.
Let $S = f_k(S^2)$ be the first thick sphere.

We know that the isootopy going  from $f_k$ to $f_{k+1}$ results in the deletion of a removable pair  in a $1$-simplex $ \gamma $ and the isootopy going from $f_k$ to $f_{k-1}$ results in the removal of another pair. 
 Both pairs
 must belong to the same $3$-simplex  and the same track of intersection with its boundary, for if there is no  such track that contains both pairs, so that one pair is in one track and the other removable pair is in another track, then the order of the isotopies can be changed and   the first peak occurs earlier.  Also if both pairs are in the same track and  one pair is not separated by the removal of the other pair,  then we can also swap the isotopies around and get a lower first peak.  It follows that the exceptional track has two removable pairs and that removing one pair disconnects the track..

The $8$-track is the only track in a tetrahedron with these properties.   

It is also the case that no removable pair  can occur anywhere else in $S_k$,  because if it did, it could be removed before doing the $k$-th step,  then put back after  the $(k+1)$-step,
again giving a lower first peak.   This means that apart from the exceptional $8$-track, $S_k$ intersects  the boundary of each $3$-simplex in $3$-tracks and $4$-tracks, making $S_k$ an almost  normal $2$-sphere.   

For the right choice of  removable pair in the $8$-track of $S_k$,  removing it gives $S_{k-1}$.   Then successively removing the ends of returning arcs by a sequence of isotopies
one ends up with a normal surface parallel to the boundary $2$-sphere $S_0$.
Choosing the other removable pair in the $8$-piece and then removing returning arcs will eventually result in the removal of all points.   Thus the isotopy sequence has just one thick 
sphere.

We have shown that $M$ is the triangulation of a $3$-sphere if and only if in the tree $D_P$ corresponding to a maximal pattern of tracks corresponding to normal $2$-spheres 
every vertex region belonging to a vertex of  $D_P$ of degree one contains a vertex of the triangulation or an almost normal track.
\end {proof}

\end {section}
\begin {section} {Spatterns and Stracks}
We now introduce a more general terminology appropriate to working with homotopies rather than isotopies.   
Many of the definitions and arguments given for injective maps of $2$-spheres into a $3$-manifold still work for piecewise linear maps.
 Let $M$ be a $3$-manifold.
Let $f : S^2\rightarrow M$  be a  general position map (see Hempel \cite {[JH76]}, Chapter 1), in which $f$ is in general position with respect to a triangulation $K$ of $M$.   An $i$-piece of $f$ is defined to be a component of $f^{-1}(\sigma)$ where $\sigma$ is an $(i+1)$-simplex of $K$. Thus a  $0$-piece is a point of $S^2$. A $1$-piece is either an scc (simple closed curve) or an arc joining two $0$-pieces.  If there are no $1$-pieces that are scc's, then each $2$-piece  has boundary that is is a union of $1$-pieces.
 One can use surgery along simple closed curves to change $f$ to a map in which there are no $1$-pieces that are scc's,  and in which every $2$-piece is a disc.
The $2$-pieces will then give  a cell decomposition (tessellation) of the $2$-sphere.   If $R$ is a $1$-piece with end points $u,v$ whose images under $f$ are in the same $1$-simplex, then,  the restriction of $f$ to $R$ is called a returning arc.

A  {\it spattern } $sP$ in $K$ is defined to be a subset of $|K|$ satisfying

\begin{itemize}

\item [(i)]  For each $2$-simplex $\sigma $ of $K$,  $sP\cap |\sigma|  $ is a union of finitely many straight lines joining distinct faces of $\sigma$.

\item [(ii)] For each $1$-simplex $\gamma$ of $K$, $sP\cap |\gamma | $ consists of finitely many points in the interior of $|\gamma |$. Each such point  belongs to exactly one straight line in each of the $2$-simplexes containing $\gamma$.

\end {itemize}

A {\it strack} is a spattern that has no proper subspatterns.   Every spattern is a union of finitely many stracks.  A strack in $T$, the tetrahedron boundary of a $3$-simplex,  is the image of a circle.
If $M$ is a $3$-manifold and $M$ is triangulated so that $M=|K|$ where $K$ is a $3$-complex,   then a spattern $sP$ in $|K^2|$ determines a {\it spatterned surface} $S$  obtained by attaching a singular disc to each strack in the boundary of each $3$-simplex $\rho $.

The image of a map $f \rightarrow M$ can be straightened to intersect the $2$-skeleton $K^2$  non-trivially in a spattern if and only if there are no $1$-pieces that are scc's or returning arcs.  
 
 Let a strack be the image $s$  of  a map $f : S^1 \rightarrow K$ which lies in a spatterned $2$-sphere.  Here $s$ will be in the boundary of a $2$-piece in a $3$-simplex $\rho $.
 Let $p,q$ be points in $\gamma \cap s$,   then $p,q$ are a removable pair if and only if starting from $p$ and proceeding along $s$ in one of the two directions possible,  $q$ is the first return to a point of $\gamma$ and this return is in ths same $2$-simplex that $s$ left $p$.  Thus in the strack shown in the left hand diagram of  Figure 13  the points in $uv$ and $wz$ are removable  but those in $vz$ and $uw$ are not removable.   The right hand diagram shows the effect of removing the pair in $wz$.   The points in $uv$ are now in different stracks.

  \begin{figure}[htbp]
\centering
\begin{tikzpicture}[scale=.6]
  \draw  (0,0) -- (0,4) ;
  \draw  (4,0) -- (4,4)--(0,4) ;
   \draw  (0,0) -- (4,4) ;
  \draw  (4,0)--(0,0) ;
  
\draw   (0,0) node {$\bullet $} ;
\draw   (4,0) node {$\bullet $} ;	
\draw  (0,4) node {$\bullet $} ;
\draw  (4,4) node {$\bullet $} ;

 \draw  (6,0)--(6,4) --(10,0)--(10,4) ; 

\draw   (6,0) node {$\bullet $} ;
\draw   (10,0) node {$\bullet $} ;
\draw  (6,4) node {$\bullet $} ;
\draw  (10,4) node {$\bullet $} ;


\draw [red]  (6,2.2) --(10,1.8) ;

\draw [red] (6,1.8)--(9,0) ;

\draw [red] (6.9, 1.1)--(6.8,1.3)--(7.2, 1.3);

\draw [red] (21.9, 1.1)--(21.8,1.3)--(22.2, 1.3);

\draw [red] (7.4, 1.4)--(7.5,1.6)--(7.8, 1.5);

\draw [red] (22.4, 1.4)--(22.5,1.6)--(22.8, 1.5);

\draw [red] (16.7, 1.1)--(16.9,1.1)--(16.9,1.3);

\draw [red] (8,4)--(10,2.2) ;
\draw [red] (1.7, 1.1)--(1.9,1.1)--(1.9,1.3);
\draw [red] (9.1, 2.8)--(9,3.1)--(9.3, 3.05) ;
\draw [red] (1, 3.5)--(1,3.1)--(1.35,3.15);
\draw [red] (0,2.2) -- (2, 4) ; 
\draw  [red] (2, 0) -- (4, 1.8) ;

\draw [red] (0,1.8)--(4,2.2) ;
\draw  [red] (1.6,2.15)--(1.9, 2)--(1.6,1.8) ;

\draw  [red] (7.6,2.2)--(7.85, 2)--(7.6,1.85) ;
\draw (6,0) --(10,0) ;
\draw  [red] (3.1, 1.15)--(3,.9)--(3.3, .9) ;

\draw (6,4) --(10,4) ;


\draw  [left] (0,0) node {$_u$} ;
\draw  [left] (6,0) node {$_u$} ;
\draw  [left] (0,4) node {$_v$} ;
\draw  [left] (6,4) node {$_v$} ;

\draw [red]  (16,4)--(16.5,1.5)--(18,0) ;
\draw [red]  (1,4)--(1.5,1.5)--(3,0) ;
\draw [red]  (7,4)--(7.1,2.9)--(8,0) ;

\draw [red]  (22,4)--(22.1,2.9)--(23,0) ;

\draw  [right ] (4,0) node {$_w$} ;
\draw  [right ] (10,0) node {$_w$} ;
\draw  [right ] (4,4) node {$_z$} ;

\draw  [right ] (10,4) node {$_z$} ;



  \draw  (15,0) -- (15,4) ;
  \draw  (19,0) -- (19,4)--(15,4) ;
   \draw  (15,0) -- (19,4) ;
  \draw  (19,0)--(15,0) ;
  
\draw   (15,0) node {$\bullet $} ;
\draw   (19,0) node {$\bullet $} ;	
\draw  (15,4) node {$\bullet $} ;
\draw  (19,4) node {$\bullet $};

\draw   (21,0) node {$\bullet $} ;
\draw   (25,0) node {$\bullet $} ;
\draw  (21,4) node {$\bullet $} ;
\draw  (25,4) node {$\bullet $} ;

\draw [blue] (16.1, 3.4)--(16,3.1)--(16.35,3.15);

\draw [red] (21,1.8)--(24,0) ;
\draw [blue] (23,4)--(21,2.2) ;
\draw [blue] (21.7, 3)--(22,3.1)--(21.9,2.8);

\draw [blue] (15,2.2) -- (17, 4) ; 
\draw  [red] (15, 1.8) -- (17, 0) ; 

\draw [red] (15.7, 1)--(16,.9)--(15.9, 1.2);

\draw (21,0) --(25,0) ;
\draw (21,0) --(21,4) ;

\draw (21,4)--(25,0) --(25,4) ;


\draw (21,4) --(25,4) ;

\draw  [left] (15,0) node {$_u$} ;
\draw  [left] (21,0) node {$_u$} ;
\draw  [left] (15,4) node {$_v$} ;
\draw  [left] (21,4) node {$_v$} ;

\draw  [right ] (19,0) node {$_w$} ;
\draw  [right ] (25,0) node {$_w$} ;
\draw  [right ] (19,4) node {$_z$} ;

\draw  [right ] (25,4) node {$_z$} ;


\end{tikzpicture}

\caption {\label 2}

\end{figure}

If $S$ is a spattern in a  $2$-complex $K$, then there is a uniquely determined underlying pattern $P$ that has the same intersection with the $1$-skeleton of $K$.
Put $W = S\cap |K^1| = P\cap |K^1|$.

Figure 14 shows a spattern in $T$ that is a union of a red 12-track and a blue 3-track, and its underlying pattern, which is also a $3$-track and a $12$-track.

 \begin{figure}[htbp]
\centering
\begin{tikzpicture}[scale=.6]


  \draw  (0,0) -- (0,4) ;
  \draw  (4,0) -- (4,4)--(0,4) ;
   \draw  (0,0) -- (4,4) ;
  \draw  (4,0)--(0,0) ;
  
\draw   (0,0) node {$\bullet $} ;
\draw   (4,0) node {$\bullet $} ;	
\draw  (0,4) node {$\bullet $} ;
\draw  (4,4) node {$\bullet $} ;

 \draw  (6,0)--(6,4) --(10,0)--(10,4) ; 

\draw   (6,0) node {$\bullet $} ;
\draw   (10,0) node {$\bullet $} ;
\draw  (6,4) node {$\bullet $} ;
\draw  (10,4) node {$\bullet $} ;

\draw [blue]  (3,0) -- (0,3.3) ;
\draw [blue]  (9,0) -- (6,3.3) ;

\draw [red] (6,1.8)--(8,0) ;

\draw [red] (8,4)--(10,2.2) ;


\draw [red] (6,2.2)--(10,2) ;
\draw [red] (6,2)--(10,1.8) ;

\draw [red] (0,2.2) -- (2, 4) ; 
\draw  [red] (2, 0) -- (4, 1.8) ; 

\draw [red] (0,2)--(4,2.2) ;

\draw [red] (0,1.8)--(4,2) ;

\draw (6,0) --(10,0) ;

\draw (6,4) --(10,4) ;

\draw  [left] (0,0) node {$u$} ;
\draw  [left] (6,0) node {$u$} ;
\draw  [left] (0,4) node {$v$} ;
\draw  [left] (6,4) node {$v$} ;

\draw  [right ] (4,0) node {$w$} ;
\draw  [right ] (10,0) node {$w$} ;
\draw  [right ] (4,4) node {$z$} ;

\draw  [right ] (10,4) node {$z$} ;


  \draw  (12,0) -- (12,4) ;
  \draw  (16,0) -- (16,4)--(12,4) ;
   \draw  (12,0) -- (16,4) ;
  \draw  (16,0)--(12,0) ;
  
\draw   (12,0) node {$\bullet $} ;
\draw   (16,0) node {$\bullet $} ;	
\draw  (12,4) node {$\bullet $} ;
\draw  (16,4) node {$\bullet $} ;

 \draw  (18,0)--(18,4) --(22,0)--(22,4) ; 

\draw   (18,0) node {$\bullet $} ;
\draw   (22,0) node {$\bullet $} ;
\draw  (18,4) node {$\bullet $} ;
\draw  (22,4) node {$\bullet $} ;

\draw [blue]  (14,0) --(13.6,1.6)--(12,1.8) ;
\draw [blue]  (20,0) --(18,1.8) ;

\draw [red] (18,2)--(21,0) ;

\draw [red] (20,4)--(22,2.2) ;


\draw [red] (18,3)--(19.76, 2.24)--(22,2) ;
\draw [red] (18,2.2)--(22,1.8) ;

\draw [red] (12,3) -- (14, 4) ; 
\draw  [red] (15, 0) -- (16, 1.8) ; 

\draw [red] (12,2.2)--(16,2.2) ;

\draw [red] (12,2)--(16,2) ;

\draw (18,0) --(22,0) ;

\draw (18,4) --(22,4) ;

\draw  [left] (12,0) node {$u$} ;
\draw  [left] (18,0) node {$u$} ;
\draw  [left] (12,4) node {$v$} ;
\draw  [left] (18,4) node {$v$} ;

\draw  [right ] (16,0) node {$w$} ;
\draw  [right ] (22,0) node {$w$} ;
\draw  [right ] (16,4) node {$z$} ;

\draw  [right ] (22,4) node {$z$} ;

\end{tikzpicture}

\caption {\label 2}
\end{figure}

Given a finite subset $F$ of the closed interval $I =[0, 1]$ and a permutation $\beta $ of $F$ there is a continuous map $\phi _I  : I \rightarrow I$ which restricts to $\beta $ on $F$ and to the identity on $\{ 0, 1\}$.  There is a homotopy between $\phi _I$ and the identity map.

Building up from such maps,  if $\nu : W \rightarrow W$ is a permutation that restricts to a permutation on $W \cap |\gamma|$
for each $1$-simplex $\gamma$,  then $\nu $ extends to a map of  the $1$-skeleton into itself which restricts to the identity on the $0$-skeleton and which is homotopic to the identity map
on $K^{(1)}$.   This map can be further extended linearly to a map of the $2$-skeleton and then, in the case of a triangulated $3$-manifold,  to a continuous map $\nu : M\rightarrow M$  which is homotopic to the identity map on $M$.  It will have the property 
that if two points on the boundary of a $2$-simplex $\sigma$ are joined by a line in $S\cap \sigma$, then they are joined by a line in $\nu S \cap \sigma$.  A spattern $S$ in $K$ is mapped to another spattern.     Lines that were uncrossed may become crossed, and lines that were crossed may become uncrossed.   The underlying pattern $P$ is not changed.
and the spatterned surface (tessellation) corresponding to the spattern is also unchanged.   A homotopy like this is called an {\it edge homotopy}.

Our proof of the Poincar\' e Conjecture is to show that a certain spattern must occur in a homotopy from the boundary of a fake ball to a constant map, and this spattern is homotopic to 
its underlying pattern by edge homotopies.


\end {section}
\begin {section} {The Poincar\' e Conjecture}
\begin {theo} A simply connected, compact $3$-manifold with no boundary is a $3$-sphere
\end {theo}
\begin {proof}
We follow the proof of the Recognition Algorithm  until the part where it is shown that an isotopy between the boundary of $M_0$  and the constant map means that $M_0$ contains an almost normal $2$-sphere.   We now have to prove that this is the case under the weaker assumption that there is a homotopy between the boundary map and the constant map.

As before
let $f :S^2 \rightarrow M$ be an injective map whose  image is the  normal $2$-sphere $\delta M_0$.
 Let $F : S^2  \times I \rightarrow M_0$
be a homotopy between $f $ and a constant map. For $t \in  I$ let $f_t : S^2 \rightarrow M_0,  f_t(s) = F(s,t)$.
 We can assume that for all but finitely many values $t, f_t$ meets the $1$-skeleton $T^1$
of the triangulation transversely and for each $t$ for which the map $f_t $ does not meet $T^1$
transversely, there is precisely one point where $f_t$ is tangential to $T^1$.
Let $t_1',t_2',\dots ,t_n'$
be the values of $t $ for which $f_t$ does not meet $T^1$
transversely and
put $t_0 = 0, t_{n+1}=1$. For $i = 1, . . . , n $ choose $t_i  \in (t_i',t_{i+1}')$
and put $f_i = f_{t_i}$.

Recall the definition of a removable pair.
Let $\gamma $ be a $1$-simplex of $M$.  Two points $p,q \in \gamma \cap f(S^2)$ are said to be {\it }removable if there is a homotopy from  $f$  to a map $f' : S^2 \rightarrow M$ such that $f(x)=f'(x)$ for every 
$x$ that is not in the interior of a simplex with $ \gamma $ as a face and $\gamma \cap f'(S^2)$ is the same as  $\gamma \cap f(S^2)$ but with $p,q$ removed.
Each step of $F$ is the removal or addition of a removable pair.

Let $W_i$ be the set of intersections of $f_i(S^2)$
with the $1$-skeleton of $T$ and let $w_i = |W_i|$.
Rearrange the weights $w_i$
into a finite non-increasing sequence. Order these sequences
lexicographically. The width of $T$ is the minimum sequence of weights as $F$ ranges over
all possible homotopies.

If F realises the width of $T$, then $F$ is said to be in thin position. Now revert to the
original ordering of the $w_i$'s.
 Clearly for each $i$ either $w_{i+1} = w_i + 2$ or $w_{i+1} = w_i-2$. Note
that $w_0$   is the number of intersections of $\delta M_0$ with 
the $1$-skeleton. For $F$ in thin position, we are particularly interested in the values  $ i$ for which $w_{i-1}= w_{i+1} = w_i-2$. For
such an $i, f_i(S^2)$ is called a thick sphere, while if $w_{i-1} = w_{i+1} = w_i + 2$, then $f_i(S^2)$
is
called a thin sphere. Our interest will be in the first thick sphere.  Since there are no removable pairs in a normal surface, $w_1 = w_0 +2$.
Let $S = f_k(S^2)$ be the first thick sphere.

We know that the homotopy going  from $f_k$ to $f_{k+1}$ results in the deletion of a removable pair  in a $1$-simplex $ \gamma $ and the homotopy going from $f_k$ to $f_{k-1}$ results in the removal of another pair. 
 Both pairs
 must belong to the same $3$-simplex  and the same strack of intersection with its boundary, for if there is no  such strack that contains both pairs, so that one pair is in one strack and the other removable pair is in another strack, then the order of the homotopies can be changed and   the first peak occurs earlier.  Also if both pairs are in the same strack and  one pair is not separated by the removal of the other pair,  then we can also swap the homotopies around and get a lower first peak.  It follows that the exceptional strack has two removable pairs and that removing one pair disconnects the strack..

As we have seen in Figure 13, there can be more than one strack with these properties.

It is also the case that no removable pair  can occur anywhere else in $S_k$,  because if it did, it could be removed before doing the $(k-1)$-th step,  then put back after  the $(k+1)$-step,
again giving a lower first peak.

Let $S_0$ be the normal $2$-sphere that is the boundary of $M_0$.    Now a normal pattern contains
no removable pair and so $S_1$ has two extra vertices $u_1,v_1$.

Now $W_0$ is the set of  vertices of the normal $2$-sphere $S_0 = \delta M_0$.
For $i=1,2,\dots , k ,\ \ \ \ W_i= W_{i-1}\cup \{ u_i,v_i\}$  where $u_i,$ and $v_i$ are points of the $1$-simplex $\gamma _i$ which are joined by a returning arc in a $2$-simplex with edge $\gamma _i$.

It is clear then that $S=f_k(S^2)$ has no returning arcs.  Also it can be assumed that it has no $1$-pieces that are scc's.    For if there was a $1$-piece that was an scc then surgery along this curve would produce two $2$-spheres,  one of which will contain the exceptional $2$-piece.   If this $2$-sphere $S'$ had fewer intersections with the $1$-skeleton,  i.e. $S'\cap K^1$ is a proper subset of $W =W_k$,  then it would contradict the choice  of $F$.   Thus $S' \cap K^1 =W$  and the other $2$-sphere will not intersect the $1$-skeleton.   All the $1$-pieces for this $2$-sphere must be scc's,  and  these can be removed from $f$ by surgeries.   Using a similar argument one can do surgery along scc's to change every $2$-piece in $S$ that is not a disc to one that is a disc.    

We can assume, then,  that $S$  is a spatterned surface.  Let $P$ be the underlying patterned surface.  It will be shown that there is a permutation of $W$,
restricting to a permutation on each intersection with a $1$-simplex for which the corresponding homotopy changes $S$ to $P$.

As a spattern is determined by its intersections with $2$-simplexes, we will consider what happens to a single $2$-simplex in the homotopy sequence.
Initially the $2$-simplex will intersect $S_0$ as in one of the diagrams of Figure 9.  The unshaded parts will be the intersection with $M_0$.   Note that since $M_0$ contains no 
vertices, each vertex is contained in a shaded region.  The intersection with $S$ will consist of the intersection with $S_0$ together  with straight lines   joining edges in the unshaded regions.  The extra intersection points with
an edge  are paired - each pair lying in an unshaded component.   

If our sequence of homotopies was a sequence of isotopies, then each isotopy would result  in an increase in shaded area.   For each $2$-simplex,   a move in which an unshaded central area   becomes shaded can occur at most once.
If this region is initially shaded then, obviously, no such move can occur.    
The other move that can occur is adding a shaded region to an unshaded region as shown in Figure 17.

It will be shown that our sequence of homotopies can be converted to a sequence of isotiopies by using the homotopies corresponding to permutations of the points of $W$ on a
$1$-simplex.

Consider now the last removal of a returning arc before $S_0$ is reached.
The only way one can end up with a normal surface with no returning arcs is for there to be a $2$-simplex as in the left hand picture of Figure 15 with a pair of points in the top left hand edge joined by a returning arc in another $2$-simplex containing that edge,  which then becomes the right hand picture on removing the returning arc.

\begin{figure}[htbp]
\centering
\begin{tikzpicture}[scale=1.2]
\fill [pink]  (4,0) --(3,0)--(3.5,.5)--(4,0);
\fill [pink] (0,0)--(.45,.45)--(.9,0) --(0,0);
\fill [pink]  (2,2)--(2.5,1.5)--(1.5,1.5)--(2,2);
\fill [pink] (2.7,1.3)--(3,1) --(2.7,0)--(2.2,0)--(2.7,1.3); ;
\fill [pink]  (1,1)--(1.6,0)--(1.2,0)--(.65,.65)--(1,1););

\fill [pink]  (10,0) --(9.2,0)--(9.4,.6)--(10,0);
\fill [pink] (6,0)--(6.45,.45)--(6.9,0) --(6,0);
\fill [pink]  (8,2)--(8.5,1.5)--(7.5,1.5)--(8,2);
\fill [pink]  (7,1)--(7.6,0)--(7.2,0)--(6.65,.65)--(7,1);

\draw [red]  (4,0) --(3,0)--(3.5,.5)--(4,0);
\draw[red] (0,0)--(.45,.45)--(.9,0) --(0,0);
\draw[red]  (2,2)--(2.5,1.5)--(1.5,1.5)--(2,2);
\draw[red] (2.7,1.3)--(3,1) --(2.7,0)--(2.2,0)--(2.7,1.3); ;
\draw[red]  (1,1)--(1.6,0)--(1.2,0)--(.65,.65)--(1,1););

\draw[red]  (10,0) --(9.2,0)--(9.4,.6)--(10,0);
\draw[red](6,0)--(6.45,.45)--(6.9,0) --(6,0);
\draw[red]  (8,2)--(8.5,1.5)--(7.5,1.5)--(8,2);
\draw[red]  (7,1)--(7.6,0)--(7.2,0)--(6.65,.65)--(7,1);
\draw[red](8.2,0)--(8.7,0)--(9,1)--(8.7,1.3)--(8.2,0);

\fill [pink](8.2,0)--(8.7,0)--(9,1)--(8.7,1.3)--(8.2,0);

\draw (0,0) -- (4,0)--(2,2) -- (0,0) ;
\draw (6,0)--(10,0)--(8,2)--(6,0) ;



\draw [red] (1.5,1.5)--(2.5,1.5) ;
\draw [red] (7.5,1.5)--(8.5,1.5) ;

\draw [red] (2.2,0)--(1.4,1.4) ;
\draw [red] (1.1,1.1)--(2.7,1.3) ;
\draw [red] (.65,.65)--(1.2,0) ;
\draw [red] (8.7,0)--(9,1) ;

\draw [red] (3,0)--(3.5,.5);
\draw [red] (.9,0)--(.45,.45);
\draw [red] (2.7,0)--(3,1);

\draw [red] (9.4,.6)--(9.2,0);
\draw [red] (6.9,0)--(6.45,.45);



\draw [red] (1.4,1.4) arc (45:225:.2) -- cycle ;

\end{tikzpicture}

\caption {\label 2}

\end{figure}

In $S_1$ the pair $u_1,v_1$ will be the ends of a returning arc in at least one $2$-simplex.   In another $2$-simplex $u_1,v_1$ will be  joined by lines in both $S_1$ and $S$ to the vertices of an edge in $S_0$.   Having such a situation in a $2$-simplex is the only way that removing $u_1, v_1$ will give a normal pattern.  Thus we are in the situation of Figure 16, 
with $i =1$, rather than Figure 17.
In the sequence of homotopies the lines joining $u_1,v_1$  to the ends $w,z$ of an edge may cross.  They can be uncrossed by transposing $u_1$ and $v_1$.
We need $\nu _1$ to do more than this.   We permute the points of $W$ that lie in the interval $[p,q]$  so that there are no points in the open interval $(u_1,v_1)$ and there are 
no lines that cross the lines from either $u_1$ or $v_1$. This will happen if the points in $W\cap [p,q]$ are permuted so that the ones joined to a point on the bottom edge are in
$[p,u_1]$ and  the points in $[q,v_1]$ are joined to points in the right hand edge.  The lines  $u_1w$ and $v_1z$ will now be lines in $P$.   This is illustrated in Figure 16(b).
If the points of $S$ in the interval $pq$ are labelled $(1,2,3,4,5)$ then $\nu _i$ takes the points to $(1,4,5,2,3)$.

\begin{figure}[htbp]
\centering
\begin{tikzpicture}[scale=1.1]
\draw [red] (1.4,1.4) arc (45:225:.2) -- cycle ;
\fill [pink]  (4,0) --(3,0)--(3.5,.5)--(4,0);
\fill [pink] (0,0)--(.45,.45)--(.9,0) --(0,0);
\fill [pink]  (2,2)--(2.5,1.5)--(1.5,1.5)--(2,2);
\fill [pink] (2.7,1.3)--(3,1) --(2.7,0)--(2.2,0)--(2.7,1.3); ;
\fill [pink]  (1,1)--(1.6,0)--(1.2,0)--(.65,.65)--(1,1););

\fill [pink]  (10,0) --(9.2,0)--(9.4,.6)--(10,0);
\fill [pink] (6,0)--(6.45,.45)--(6.9,0) --(6,0);
\fill [pink]  (8,2)--(8.5,1.5)--(7.5,1.5)--(8,2);
\fill [pink]  (7,1)--(7.6,0)--(7.2,0)--(6.65,.65)--(7,1);
\fill [pink] (8.2,0)--(8.7,0)--(9,1)--(8.7,1.3)--(7.3,1.3) -- (7.2,1.2)--(8.2,0);
\fill [pink] (7.3,1.3) arc (45:225:.08) -- cycle ;
\draw [red]  (4,0) --(3,0)--(3.5,.5)--(4,0);
\draw[red] (0,0)--(.45,.45)--(.9,0) --(0,0);
\draw[red]  (2,2)--(2.5,1.5)--(1.5,1.5)--(2,2);
\draw[red] (2.7,1.3)--(3,1) --(2.7,0)--(2.2,0)--(2.7,1.3); ;
\draw[red]  (1,1)--(1.6,0)--(1.2,0)--(.65,.65)--(1,1););

\draw [red] (7.3,1.3) arc (45:225:.08) -- cycle ;

\draw[red]  (10,0) --(9.2,0)--(9.4,.6)--(10,0);
\draw[red](6,0)--(6.45,.45)--(6.9,0) --(6,0);
\draw[red]  (8,2)--(8.5,1.5)--(7.5,1.5)--(8,2);
\draw[red]  (7,1)--(7.6,0)--(7.2,0)--(6.65,.65)--(7,1);
\draw[red](8.2,0)--(8.7,0)--(9,1)--(8.7,1.3)--(7.3,1.3) -- (7.2,1.2)--(8.2,0);

\draw (0,0) -- (4,0)--(2,2) -- (0,0) ;
\draw (6,0)--(10,0)--(8,2)--(6,0) ;



\draw [red] (1.5,1.5)--(2.5,1.5) ;
\draw [red] (7.5,1.5)--(8.5,1.5) ;

\draw [red] (2.2,0)--(1.4,1.4) ;
\draw [red] (1.1,1.1)--(2.7,1.3) ;
\draw [red] (.65,.65)--(1.2,0) ;
\draw [red] (8.7,0)--(9,1) ;

\draw [red] (3,0)--(3.5,.5);
\draw [red] (.9,0)--(.45,.45);
\draw [red] (2.7,0)--(3,1);

\draw [red] (9.4,.6)--(9.2,0);
\draw [red] (6.9,0)--(6.45,.45);
\draw [above] (7.25,1.25) node {$_{v_i}$};
\draw [above] (7.1,1.1) node {$_{u_i}$};
\draw [below] (2.2,0) node {$_w$};
\draw [below] (8.2,0) node {$_w$};
\draw [above] (2.75,1.25) node {$_z$};
\draw [above] (8.75,1.25) node {$_z$};

\draw [above] (1.45,1.45) node {$_q$};
\draw [above] (1.3,1.3) node {$_{u_i}$};
\draw [above] (1.1,1.08) node {$_{v_i}$};
\draw [above] (.925,.925) node {$_p$};

\draw [above] (7.45,1.45) node {$_q$};
\draw [above] (6.95,.95) node {$_p$};


\end{tikzpicture}


\end{figure}

\begin{figure}[htbp]
\centering
\begin{tikzpicture}[scale=1.4]

\fill [pink]  (2.5,1.5)--(2.3,1.7)--(1.7,1.7)--(1.5,1.5)--cycle;
\fill [pink] (2.7,1.3)--(2.9,1.1) --(2.5,0)--(2.2,0)--(2.7,1.3); ;
\fill [pink]  (1,1)--(1.6,0)--(1.3,0)--(.75,.75)--(1,1););

\fill [pink]  (7,1)--(7.6,0)--(7.3,0)--(6.75,.75)--(7,1);
\fill [pink] (8.2,0)--(8.5,0)--(8.9,1.1)--(8.7,1.3)--(7.3,1.3) -- (7.2,1.2)--(8.2,0);

\fill [pink]  (8.5,1.5)--(8.3,1.7)--(7.7,1.7)--(7.5,1.5)--cycle;
\fill [pink] (2.7,1.3)--(2.9,1.1) --(2.5,0)--(2.2,0)--(2.7,1.3); ;
\fill [pink]  (1,1)--(1.6,0)--(1.3,0)--(.75,.75)--(1,1););

\draw[red] (2.7,1.3)--(2.9,1.1);
\draw[red]  (1,1)--(1.6,0)--(1.3,0);
\draw [red] (.65,.65)--(1,1);
\draw[thick,red](8.7,1.3)--(7.3,1.3) -- (7.2,1.2)--(8.2,0);

\draw (2.2,1.8)--(3,1);
\draw (8.2,1.8)--(9,1);
\draw (1.8,1.8) -- (.5,.5) ;
\draw (7.8, 1.8)--(6.5,.5);
\draw (1,0)--(2.7,0);
\draw (7,0)--(8.7,0);

\draw [red] (1.5,1.5)--(2.5,1.5) ;
\draw [red] (7.5,1.5)--(8.5,1.5) ;

\draw [red] (2.2,0)--(1.4,1.4) ;
\draw [red] (1.1,1.1)--(2.7,1.3) ;
\draw [blue] (1.9,0)--(1.3,1.3);
\draw [blue] (1.05,1.05)--(2.05,0);
\draw [brown] (2.55, 1.45)--(1.2,1.2);
\draw [blue] (7.9,0)--(7.1,1.1);
\draw [blue] (8.05,0)--(7.05,1.05);
\draw [brown] (8.55, 1.45)--(7.4,1.4);

\draw [above] (7.25,1.25) node {$_{v_i}$};
\draw [above] (7.1,1.1) node {$_{u_i}$};
\draw [below] (2.2,0) node {$_w$};
\draw [below] (8.2,0) node {$_w$};
\draw [above] (2.75,1.25) node {$_z$};
\draw [above] (8.75,1.25) node {$_z$};

\draw [above] (1.45,1.45) node {$_q$};
\draw (1.35,1.45) node {$_{u_i}$};
\draw (1.05,1.15) node {$_{v_i}$};
\draw [above] (.925,.925) node {$_p$};

\draw [above] (7.45,1.45) node {$_q$};
\draw [above] (6.95,.95) node {$_p$};

\fill [white](0,0)--(.45,.45)--(.9,0) --(0,0);

\draw [thick,red] (7.3,1.3) arc (45:225:.08) -- cycle ;
\fill [pink] (7.3,1.3) arc (45:225:.08) -- cycle ;

\draw [red] (1.4,1.4) arc (45:225:.2) -- cycle ;

\draw (2,-.5) node {(a)} ;
\draw (8,-.5) node {(b)} ;

\end{tikzpicture}

\caption {\label 2}

\end{figure}

We use an induction argument for defining $\nu _i$ for $2\leq i\leq k$.

Our aim is to show that for each $1$-simplex $\gamma $ we can permute the finite set $\gamma \cap S$ in such a way that under the associated homotopies the spattterned $2$-sphere $S$ becomes the underlying patterned $2$-sphere $P$. It will then be the case that
 $F$ becomes an isotopy.  
 
  Let $\gamma _i$  be the $1$-simplex containing $u_i$ and $v_i$.
 
 Our induction hypothesis is that there are permutations $\nu _i$, $\mu _i$ of $W$ for $i=1, \dots , k$,  such that $\nu _i$ restricts to a permutation on the points of $W\cap \gamma _i$
 and is the identity on the other points of $W$,
 and so that  the homotopy associated with $\mu _i = \nu _i\mu _{i-1}$   moves the points $u_j $ and $v_j$ for $j<i$ so that  any line joining two such points is moved to a line in $P$, and so that the lines joining $u_i$ to $u_j$  and joining $v_i$ to $v_j$ are lines in $P$ and they do not cross any other line  from $\gamma _j$ to $\gamma _i$.
 
We have defined $\nu _1$.  Let $\mu _1 = \nu _1$.

 We now define $\nu _i $   and $\mu _i$.    
 
 Going from $S_k$ to $S_0$, each step apart from the first is the removal of a returning arc with the possible creation of other returning arcs.  Thus going from $S_0$ to $S_k$, each
 step apart from the last is the creation of a returning arc joining $u_i$ and $v_i$ with the possible removal of other returning arcs.   
 
  In at least one of the $2$-simplexes containing $\gamma _i$ as a face there are lines in $S$ joining $u_j$ to $u_i$ for some $j <i$ or there is a situation as in Figure 16, in which $u_i,v_i$ are joined to the vertices of an edge of $S_j$ for $j < i$.  This is because there cannot be a returning arc joining $u_i$ and $v_i$ in every $2$-simplex containing $\gamma _i$, as if there were, then $S_i$ would not be connected.  If $u_i$ and $v_i$ are joined to points $u_j$ and $v_j$ for $j<i$ then $u_j$ and $v_j$ are the ends of a returning arc in $S_l$ for every $l,   j\leq l <i$.

 In Figure 17  we  see what happens in a $2$-simplex $\sigma _i$ containing $\gamma _i$, if the lines  joining $u_i$ and $v_i$ come from the pair $u_j,v_j$.  

 We permute the points of $W$ that lie in the interval $[r,s]$  so that there are no points in the smaller interval $[u_i,v_i]$ and there are 
no lines that cross the lines from either $u_i$ or $v_i$. This will happen if the points in $W\cap [r,s]$ are permuted so that the ones joined to a point in $[x,u_j]$ are in
$[r,u_i]$ and  the the points in $[s,v_i]$ are joined to points in $[v_j,y]$.  The lines  $u_iu_j$ and $v_iv_j$ will now be lines in $P$.
This is also illustrated in Figure 17.
Note that every point in the interval $xy  $ apart from $u_j$ and $v_j$  has label bigger than $j$ and every label in  the interval $rs$ apart from $u_i$ and $v_i$ has label bigger than $i$.

A similar argument can be used if both pairs are in the central region.

At the end of the induction a connected subgraph, containing every vertex  of the $1$-skeleton of $S$, has been moved to the position it should have in $P$.  In fact we could have chosen the lines of any $2$-simplex to be in this connected subgraph.  So this means that  $P=\mu _kS$ and  so $P$ determines a patterned $2$-sphere.  

We now know that $\mu _k S$ is a patterned surface.    
  All the $2$-pieces, apart from the exceptional one, intersect each $1$-simplex at most once and so  are $3$-sided or $4$-sided.
  The $8$-track shown in Figure 12  is the only possibility for the exceptional $2$-piece.  Thus $S$ has become  an almost normal $2$-sphere and we have a proof of the Poincar\' e Conjecture.
  
      After applying $\mu $,   for $j>k$  each step of the homotopy becomes an isotopy in which a removable pair of adjacent points  is removed.    There is now a new labelling of $W$ in which every point receives a label $j$, where $n \leq j \leq n$.   The pair of points labelled
  $j$ will be a removable pair in $f_j(S^2)$..  In fact the pair of points will be  joined by a returning arc for $j >k$.
  
   \end {proof}
 \eject   
\begin{figure}[htbp]
\centering
\begin{tikzpicture}[scale=1.2]

\draw [red,thick] (6.9,.9) arc (45:225:.19) -- cycle ;
\fill [pink] (6.9,.9) arc (45:225:.19) -- cycle ;
\draw [red] (.9,.9) arc (45:225:.2) -- cycle ;

\fill [pink](1.2,0) arc (180: 0:.2)--(1.2,0) ;
\fill [pink]  (4,0) --(3,0)--(3.5,.5)--(4,0);
\fill [pink] (0,0)--(.45,.45)--(.9,0) --(0,0);
\fill [pink]  (2,2)--(2.5,1.5)--(1.5,1.5)--(2,2);
\fill [pink] (2.7,1.3)--(3,1) --(2.7,0)--(2.2,0)--(1.1,1.1)--(1.3,1.3)--(2.7,1.3); ;

\fill [pink]  (10,0) --(9.2,0)--(9.4,.6)--(10,0);
\fill [pink] (6,0)--(6.45,.45)--(6.9,0) --(6,0);
\fill [pink]  (8,2)--(8.5,1.5)--(7.5,1.5)--(8,2);
\fill [pink]  (6.9,.9)--(7.6,0)--(7.2,0)--(6.65,.65)--(6.9,.9);
\fill [pink] (8.2,0)--(8.7,0)--(9,1)--(8.7,1.3)--(7.3,1.3) -- (7.1,1.1)--(8.2,0);

\draw[red]  (4,0) --(3,0)--(3.5,.5)--(4,0);
\draw [red](0,0)--(.45,.45)--(.9,0) --(0,0);
\draw[red] (2,2)--(2.5,1.5)--(1.5,1.5)--(2,2);
\draw[red] (2.7,1.3)--(3,1) --(2.7,0)--(2.2,0)--(1.1,1.1)--(1.3,1.3)--(2.7,1.3); ;

\draw[red] (10,0) --(9.2,0)--(9.4,.6)--(10,0);
\draw[red] (6,0)--(6.45,.45)--(6.9,0) --(6,0);
\draw[red] (8,2)--(8.5,1.5)--(7.5,1.5)--(8,2);
\draw[red] (6.9,.9)--(7.6,0)--(7.2,0)--(6.65,.65)--(6.9,.9);
\draw[red](8.2,0)--(8.7,0)--(9,1)--(8.7,1.3)--(7.3,1.3) -- (7.1,1.1)--(8.2,0);

\draw (0,0) -- (4,0)--(2,2) -- (0,0) ;
\draw (6,0)--(10,0)--(8,2)--(6,0) ;



\draw [red] (1.5,1.5)--(2.5,1.5) ;
\draw [red] (7.5,1.5)--(8.5,1.5) ;
\draw [red] (1.2,0)--(.9,.9);
\draw [red] (1.6,0)--(.6,.6);
\draw [red] (8.7,0)--(9,1) ;

\draw [red] (3,0)--(3.5,.5);
\draw [red] (.9,0)--(.45,.45);
\draw [red] (2.7,0)--(3,1);

\draw [red] (9.4,.6)--(9.2,0);
\draw [red] (6.9,0)--(6.45,.45);

\draw [above] (1.1,1.1) node {$_s$};
\draw [above] (7.1,1.1) node {$_s$};
\draw [below] (2.2,0) node {$_y$};
\draw [below] (7.6,0) node {$_{v_j}$};
\draw [above] (.4,.4) node {$_r$};
\draw [below] (8.2,0) node {$_y$};
\draw [above] (6.4,.4) node {$_r$};
\draw [below] (1.6,0) node {$_{v_j}$};
\draw [below] (1.2,0) node {$_{u_j}$};
\draw [above] (.85,.85) node {$_{u_i}$};
\draw [above] (.55,.55) node {$_{v_i}$};

\draw [above] (6.6,.6) node {$_{u_i}$};
\draw [above] (6.85,.85) node {$_{v_i}$};
\draw [below] (7.2,0) node {$_{u_j}$};
\draw [below] (6.9,0) node {$_x$};
\draw [below] (.9,0) node {$_x$};

\end{tikzpicture}


\end{figure}

\begin{figure}[htbp]
\centering
\begin{tikzpicture}[scale=1.1]
\fill [pink] (0,0)--(.4,0)--(.4,2)--(0,2)--cycle;
\draw (0,0)--(4,0) ;
\draw (0,2)--(4,2);
\draw [red,thick] (.4,0)--(.4,2);
\fill [pink] (3.6,0)--(4,0)--(4,2)--(3.6,2)--cycle;
\fill [pink] (7.7,2.2)--(8,2.2)--(8.1,2)--(7.6,2)--cycle;

\draw [red,thick] (3.6,0)--(3.6,2);
\draw [red,thick] (2,0)--(3,2);
\draw[red,thick] (2.4,0)--(1,2);

\fill [pink] (6,0)--(6.4,0)--(6.4,2)--(6,2)--cycle;
\draw (6,0)--(10,0) ;
\draw (6,2)--(10,2);
\draw [red,thick] (6.4,0)--(6.4,2);
\fill [pink] (9.6,0)--(10,0)--(10,2)--(9.6,2)--cycle;
\fill [pink] (8,0)--(8.4,0)--(8.1,2)--(7.6,2)--cycle;

\draw [red,thick] (9.6,0)--(9.6,2);
\draw [red,thick] (8,0)--(7.6,2);
\draw[red,thick] (8.4,0)--(8.1,2);
\draw [blue,thick] (1,0)--(.8,2);
\draw [blue,thick] (7,0)--(6.8,2);
\draw [blue,thick] (.5,0)--(1.6,2) ;
\draw [blue,thick] (6.5,0)--(7,2);
\draw [green,thick] (1.5,0) -- (3.2,2)  ;
\draw [green, thick] (7.5,0) --(7.3,2);
\draw [brown, thick] (3.2,0)--(1.3,2);
\draw [brown, thick] (9.2,0)--(9,2);
\draw [brown,thick] (2.7, 0)--(2.1,2);
\draw [brown,thick] (8.7,0)--(9.2,2);
\draw [red, thick] (1,2)--(1.2,2.2)--(2.8,2.2)--(3,2);
\draw [red, thick] (7.6,2)--(7.7,2.2)--(8,2.2)--(8.1,2);
\draw [below] (2.4,0) node {$v_j$};
\draw [above] (3,2) node {$u_i$};
\draw [above] (1,2) node {$v_i$};
\draw [below] (8,0) node {$u_j$};
\draw [below] (2,0) node {$u_j$};

\draw [below] (8.4,0) node {$v_j$};
\draw [above] (7.5,2) node {$u_i$};
\draw [above] (8.1,2) node {$v_i$};
\draw [below] (.4,0) node {$x$};
\draw [below] (3.6,0) node {$y$};
\draw [above] (.4,2) node {$r$};
\draw [above] (3.6, 2) node {$s$};

\draw [below] (6.4,0) node {$x$};
\draw [below] (9.6,0) node {$y$};
\draw [above] (6.4,2) node {$r$};
\draw [above] (9.6, 2) node {$s$};

\end{tikzpicture}

\caption {\label 2}

\end{figure}


If the points of $S$ in the interval $rs$ are labelled $(1,2,3,4,5,6,7)$ then $\nu _i$ takes the points to $(1,5,6,2,7,4,3)$.

\end {section}

\end {document}